\documentclass[a4paper, 11pt]{amsart}

\usepackage{graphicx}
\usepackage{amsthm, amssymb, amsmath}
\usepackage{hyperref}
\makeatletter
\newcommand{\finalremark}{%
  \par\addvspace{.5\linespacing}
  \refstepcounter{subsection}
  \noindent\textbf{\thesubsection.}~~\ignorespaces
}
\makeatother

\newtheorem{theorem}{Theorem}
\newtheorem{lemma}[theorem]{Lemma}
\newtheorem{corollary}[theorem]{Corollary}
\newtheorem{proposition}[theorem]{Proposition}

\usepackage[backend=biber,style=alphabetic, sorting=none]{biblatex}
\RequirePackage{doi}
\newcommand{\Hadamard}{\odot}
\newcommand{\KK}{\tilde{K}}
\newcommand{\XX}{\mathbb{X}}
\newcommand{\YY}{\mathbb{Y}}
\newcommand{\HH}{\mathcal{H}}
\newcommand{\EE}{\mathcal{E}}
\newcommand{\multi}[1]{{\bf{#1}}}
\newcommand{\iverson}[1]{\left[\text{``}#1\text{''}\right]}
\newcommand{\Par}[1]{\operatorname{Par}\left(#1\right)}
\newcommand{\sub}[1]{\operatorname{sub}\left(#1\right)}
\newcommand{\substar}[1]{\operatorname{sub}^\ast\left(#1\right)}
\newcommand{\norm}[1]{\|#1\|}
\newcommand{\norma}[1]{{#1}^\star}
\newcommand{\NZ}{\sigma_\ast}
\newcommand{\NZn}{(\NZ)^n}
\newcommand{\killer}{\mathcal{L}}
\newcommand{\symbf}{f}
\newcommand{\symbfdh}{\check{\symbf}}
\newcommand{\per}{\operatorname{per}}
\newcommand{\polyGessel}{\triangle}
\newcommand{\polysub}{\triangledown}
\newcommand{\polyfiber}{\square}
\newcommand{\polyfiberg}{\blacksquare}
\newcommand{\polyrook}{\blacktriangle}
\newcommand{\polysubg}{\blacktriangledown}
\newcommand{\setst}[2]{\left\lbrace #1 \;\Big|\;  #2 \right\rbrace}

\newcommand{\set}[1]{\{#1\}}
\newcommand{\vof}[1]{\widehat{#1}}
\newcommand{\wof}[1]{\widetilde{#1}}

\newcommand{\NN}{\mathbb{N}}
\newcommand{\RR}{\mathbb{R}}
\newcommand{\FA}{\mathbb{A}_m}
\newcommand{\RP}[1]{\mathcal{R}(#1)}

\newcommand{\rp}[2]{\operatorname{R}_{#1}(#2)}
\newcommand{\frp}[2]{\operatorname{R}^\ast_{#1}(#2)}
\newcommand{\scalar}[2]{\left\langle #1 \;\Big|\; #2 \right \rangle}

\usepackage{listings}
\usepackage{xcolor}

\definecolor{codegreen}{rgb}{0,0.6,0}
\definecolor{codegray}{rgb}{0.5,0.5,0.5}
\definecolor{codepurple}{rgb}{0.58,0,0.82}
\definecolor{backcolour}{rgb}{0.95,0.95,0.92}

\lstdefinestyle{mystyle}{
    backgroundcolor=\color{backcolour},   
    commentstyle=\color{codegreen},
    keywordstyle=\color{magenta},
    numberstyle=\tiny\color{codegray},
    stringstyle=\color{codepurple},
    basicstyle=\ttfamily\footnotesize,
    breakatwhitespace=false,         
    breaklines=true,                 
    captionpos=b,                    
    keepspaces=true,                 
    numbers=left,                    
    numbersep=5pt,                  
    showspaces=false,                
    showstringspaces=false,
    showtabs=false,                  
    tabsize=2
}

\begin{document}


\title[Counting partial Latin rectangles]{Counting partial Latin rectangles and tridimensional rook placements with multisymmetric functions}

\author{Emmanuel Briand} 
\address{Emmanuel Briand, Departamento Matemática Aplicada 1, Universidad de Sevilla}
\email{ebriand@us.es}

\thanks{This research was supported by Grant PID2020-117843GB-I00 funded by MCIN/AEI/10.13039/501100011033.}

\begin{abstract}
  We generalize Gessel's Formula for the number of Latin rectangles  to partial Latin rectangles and  non-attacking rook placements in a tridimensional chessboard. 
  We also derive explicit short formulas for the generating series of the numbers of  non-attacking rook placements on a chessboard with 2 or 3 levels. These series also count  partial Latin rectangles with $2$ or $3$ rows. 
  
  The results are obtained following methods developed by MacMahon and Gessel for counting Latin squares and Latin rectangles, by means of scalar products of multisymmetric functions.
\end{abstract}



\maketitle


\section{Introduction}

\subsection{Overview}

Let $m$, $n$ and $r$ be nonnegative integers.
A \emph{3-dimensional non-attacking rook placement} (for short: 3D rook placement) \emph{on a  $m\times n \times r$ chessboard} is a subset $S$ of $\{1,2,\ldots,m\}\times \{1,2,\ldots,n\} \times \{1,2,\ldots, r\}$ such that if $(i,j,k)$ and $(i', j',k')$ are two distinct elements of $S$, then $i\neq i'$, and $j\neq j'$ and $k\neq k'$.

This generalizes the (2-dimensional) non-attacking rook placements on rectangular boards, that have been studied in great detail (see \cite[ch. 2]{Stanley} and the references therein).
In the tridimensional generalization we consider, the rooks can move in any of the three directions parallel to the axes of coordinates. (Other generalizations have been considered in \cite{Other3DRooks}, where the rooks can move freely  in slices parallel to  the planes of coordinates. This is a totally different story).

\emph{Partial Latin rectangles} are avatars of the 3D rook placements.  
A \emph{partial Latin rectangle} (PLR) \emph{of format $m\times n$ with $r$ symbols} is an $m\times n$ array with some boxes filled with integers from $1$ to $r$, so that no number appears twice in the same row, and no number appears twice in the same column. A PLR $L$ corresponds to the 3D rook placement $S$ such that $(i,j,k)$ is in $S$ if and only if $L$ has an $i$ in position $(j,k)$.
Partial Latin rectangles are a recently introduced generalization of \emph{Latin  rectangles} (which are PLRs of format $n\times r$ with $r$ symbols and no empty cell), which themselves generalize \emph{Latin squares} (Latin rectangles of format $r \times r$). Enumerating Latin squares and Latin rectangles are very classical problems, on which there exists an extensive literature (see \cite{StonesThesis} for references). The problem of counting PLRs has been studied systematically by Stones and Falc\'on (\cite{FalconStones2015, FalconStones2020}).

Let $\RP{m,n,r}$ be the set of all rook placements on a $m\times n \times r$ chessboard. 
For all $d \ge 0$, let  $\RP{m,n,r|d}$ be the set of all rook placements with cardinality $d$ on a  $m\times n \times r$ chessboard. 
Let us define also
\[
\rp{m,n,r}{x} = \sum_d \# \RP{m,n,r|d} x^d. 
\]
This is the (ordinary) generating series for the rook placements on a chessboard of fixed format $m\times n \times r$, in function of their number of elements.
This generalizes to our 3D setting the \emph{rook polynomial} from the (2D) rook theory. We keep for it the name \emph{rook polynomial}.

The numbers $\RP{m,n,r}$, $\RP{m,n,r|d}$ and the polynomials $\rp{m,n,r}{x}$ are obviously invariant under permutations of $(m,n,r)$. Therefore, they count not only PLRs of format $m\times n$ with $r$ symbols but also PLRs of format $n\times r$ with $m$ symbols, etc...

The total number of rook placements $\# \RP{m,n,r}$ on the chessboard is obtained by specializing at $x=1$. When $m \le n$, the leading coefficient of the rook polynomial $\rp{m,n,n}{x}$ is the number of Latin rectangles with $m$ rows. Thus this rook polynomial  interpolates between the total number of rook placements  ($x=1$) on the $m\times n \times n$ chessboard, and the number of Latin rectangles with $m$ rows and $n$ symbols (obtained with $x\to \infty$, i.e. as the top coefficient, when $m \le n$).

As an explicit example, here are the rook polynomials for a $3 \times n \times r$ chessboard, with $n$ and $r$ from $0$ to $3$:

\[
\begin{array}{llll}
1 & 1 & 1 & 1\\
1 & 3 x + 1 & 6 x^{2} + 6 x + 1 &  6 x^{3} + 18 x^{2} + 9 x + 1 \\
1 & 6 x^{2} + 6 x + 1 & 18 x^{4} + 48 x^{3} + 42 x^{2} + 12 x + 1 &  \rp{3,2,3}{x}\\
1 & 6 x^{3} + 18 x^{2} + 9 x + 1 & \rp{3,3,2}{x} &  \rp{3,3,3}{x}
\end{array}
\]
where
\[
 \rp{3,3,2}{x}=\rp{3,2,3}{x} =  12 x^{6} + 108 x^{5} + 270 x^{4} + 264 x^{3} + 108 x^{2} + 18 x + 1,
 \]
 and
 \begin{multline*}
   \rp{3,3,3}{x}= 12 x^{9} + 108 x^{8} + 756 x^{7} + 2412 x^{6} + 3834 x^{5}\\
   + 3078 x^{4} + 1278 x^{3} + 270 x^{2} + 27 x + 1.
 \end{multline*}
 In particular, we read from $R_{3,3,3}(x)$ that there are 12 Latin squares of format $3 \times 3$ (coefficient of $x^{9})$, and that the total number of rook placements in a  $3\times 3 \times 3$ chessboard is $11776$  (evaluation at $x=1$).

 MacMahon provided two related methods to enumerate Latin rectangles.
 His first method \cite[Vol. I, sec. V, ch. III]{MacMahon:book} is based on extracting a coefficient from a power of the permanent.
 His second method \cite[Vol. II, sec. XI, ch. V]{MacMahon:book} is based on calculations of multisymmetric functions (a generalization of symmetric functions where each variable $x_i$ is a vector with $m$ components $x_{i,j}$).  In \cite{Gessel:enumerative}, Gessel updated and extended MacMahon's second method.

\subsection{Main results}

 In this work, we apply the tools developed by MacMahon and Gessel to enumerate partial Latin rectangles (and, at the same time, 3D Rook Placements) with multisymmetric functions. Our first result is the generalization of a formula, due to Gessel, for the number of Latin rectangles, to PLRs. This formula was proved in the language of graphs in  \cite{Gessel:Latin}. But in  the other paper \cite{Gessel:enumerative}, published the same year ,  Gessel shows the formula could also be derived in the framework of multisymmetric functions. We provide this derivation, as well as the derivation of our generalization to PLRs.
 Next, we consider  the following  exponential generating series for the rook polynomials attached to  a  $m\times n \times r$ chessboard, for fixed $m$:
\[
g_m(z,t,x)=\sum_{n, r} \rp{m,n,r}{x} \frac{z^n}{n!} \frac{t^r}{r!}
=
\sum_{n, r, d} \# \RP{m,n,r|d}  x^d \frac{z^n}{n!} \frac{t^r}{r!}.
\]
Our second result is an explicit, short expression for $g_2$, and for $g_3$.

\subsection{Structure of the paper}

The organization of the paper is the following: sections  \ref{sec:Gessel} and  \ref{sec:results} state our results on, respectively, the generalization of Gessel's Formula , and the generating series $g_m$. Section  \ref{sec:multisymmetric} introduces the multisymmetric functions, explains their use by MacMahon and Gessel, and presents the proofs of our results.

More precisely, in  Section \ref{sec:Gessel}, we revisit Gessel's Formula for Latin rectangles.
We also give a new presentation of the formula as a ``symbolic method''.
We explain how this symbolic method allows to derive  easily known formulas  for the numbers of Latin rectangles with $2$, $3$ or $4$ rows, found in \cite{BogartLongyear, Yamamoto49, Pranesachar23, Pranesachar:chromatic}.
Then, we state our generalization of  Gessel's Formula to 3D rook polynomials. We also generalize Gessel's Formula  to what we call ``relaxed Latin rectangles'', whose numbers are the values of the chromatic polynomial of the line graph of $K_{m,n}$, considered in \cite{Pranesachar:chromatic}.
Finally, we clarify the nature of the domains of summation of Gessel formula and its generalizations, by describing them as the sets of lattice points of  polytopes.

In Section \ref{sec:results}, we consider the series $g_m$ and state our formulas for $g_2$ and $g_3$. We have implemented these formulas in the  notebook \cite{Notebook} for SageMath \cite{Sagemath}, that the reader should find of interest. As an aside, we derive generating series for the values of the chromatic polynomials of the rook's graphs with 2 or 3 rows.

Section \ref{sec:multisymmetric} is devoted to the proofs of the previous results.   We first introduce the multisymmetric functions. We derive Gessel's Formula in the setting of multisymmetric functions.
Next, we explain the  modifications that lead to the analogue of Gessel's Formula for 3D rook polynomials, and the formulas for the generating series  $g_2$ and $g_3$ introduced in the previous sections.

There are other formulas for counting Latin rectangles, due to Doyle \cite{Doyle} (see also \cite{DeGennaro, StonesComputing}), but they do not generalize to PLRs (see remark \ref{rem:Doyle}). Yet, we found interesting to look for a derivation of Doyle's Formulas in the framework of multisymmetric functions. This derivation is presented in Appendix \ref{sec:Doyle}. It is not simpler than the original. Yet, it may interest the symmetric functions enthusiast, as analogues of classical operators on symmetric  functions show up.

\subsection{Relation with \cite{Briand:extendedabstract}}

This article expands widely the extended abstract \cite{Briand:extendedabstract}, where some of the results presented here were announced.
The present paper contains much new material, not included in the extended abstract: the whole Section \ref{sec:Gessel}
on Gessel's Formula  and its generalization;
the detailed calculations and proofs with multisymmetric  functions that appear in Section \ref{sec:multisymmetric}; and Appendix \ref{sec:Doyle} on  Doyle's Formula.



\section{Gessel's Formula  for counting Latin rectangles, and its generalization for 3D rook polynomials}\label{sec:Gessel}

In this  section, we review some  explicit formulas, from the literature, for the number $L_{m,n}$ of Latin rectangles with $m$ rows and $n$ columns, for fixed $m$, equal to $1$, $2$, $3$ and $4$. The general form (for all $m$) of these formulas was written down by Gessel, and we review also Gessel's Formula.  Then we generalize Gessel's Formula to rook polynomials $\rp{m,n,r}{x}$.

\subsection{Formulas for the numbers of Latin rectangles}

Let us review known formulas for $L_{m,n}$ for small $m$.

Of course, $L_{1,n}=n!$, and the $\frac{L_{2,n}}{n!}$ are the derangements numbers, whence
\[
L_{2,n} = (n!)^2  \sum_{k=0}^n \frac{(-1)^k}{k!}.
\]

For $m=3$, we have the following formula due to Bogart and Longyear \cite[Formula (11)]{BogartLongyear}:
\begin{equation}\label{eq:BogartLongyear}
L_{3,n}
=
  \sum_{a,i,j,k,s\atop \text{s.t.} a+i+j+k+s=n} (-1)^{i+j+k}  \frac{(a+i)! (a+j)! (a+k)!}{i! j! k!} \frac{n!^2}{a!^2} \frac{2^s}{s!}.
\end{equation}
Using the identity
\[
\sum_{i,j,k\atop i+j+k=d}
\binom{a+i}{i}
\binom{a+j}{j}
\binom{a+k}{k}
=
\binom{3a+d+2}{d},
\]
(that can be proved easily with ``dots and dividers''),
the formula further simplifies into
\[
L_{3,n}
=
n!^2
\sum_{a,d,s\atop a+d+s=n}
(-1)^{d}
\binom{3a +d+2}{d} a!  \frac{2^s}{s!}.
\]
This formula appears in  \cite[theorem 4.2]{Pranesachar:chromatic} and \cite{Pranesachar23} 
According to Stones \cite{StonesComputing}, the formula is  also
attributed by Riordan \cite{Riordan} to Yamamoto \cite{Yamamoto49}.

A formula of the same kind for $m=4$ can be found in \cite[Theorem 4.3]{Pranesachar:chromatic} by setting $r=n$ (for $r\neq n$, their formula is for ``relaxed Latin rectangles'', that are a little bit more general than Latin rectangles, see Section \ref{relaxed}). 
This is a large formula involving many factorials and binomial coefficients, that we will not reproduce it here. Let us just mention that 18 indices are involved.


\subsection{Gessel's Formula }

\subsubsection{Notations}

In  order to present Gessel's Formula, we need to introduce some notations.

Given  a finite set $X$,
a \emph{multiset of elements of $X$} is a family $\nu$ of nonnegative integers $(\nu_x)_{x \in X}$ indexed by $X$. The set of all multisets of elements of $X$ is $\NN^X$. 
When $\sum_{x \in  X} \nu_x = n$, we say that this multiset is a \emph{$n$-multiset}, and that $n$ is the \emph{size} of $\nu$. Then, we  write $|\nu|$ for $n$ . Given such a $n$-multiset, we set
\[
\binom{n}{\nu} = \frac{n!}{\displaystyle \prod_{x\in X} \nu_x!}
\]
(multinomial coefficient).

Let $[m] = \{1,2,\ldots, m\}$.
Let $\sub{m}$ be the set of all subsets of $[m]$, and $\substar{m}$ be the set of the \emph{non-empty} subsets of $[m]$.

For any subset $A$ of $[m]$, we note $\tau_A = (\#A-1)!$.
For any multiset $\phi$ of non-empty subsets of $[m]$, we note $\norm{\phi}=\sum_{A \subset [m]} \phi_A \cdot \#A$.

Let $\Par{m}$ be the set of all set partitions of $[m]$.
The notation ``$\pi \vDash X$'' means ``$\pi$ is a set partition of $X$''.

\subsubsection{Gessel's Formula}

Any multiset $\nu$ of set partitions of $[m]$ induces a multiset $\vof{\nu}$ of non-empty subsets of $[m]$, defined by:
\[
\vof{\nu}_A = \sum_{\pi \vDash [m]\atop \text{s.t. } A\in \pi} \nu_\pi.
\]

\begin{theorem}[\cite{Gessel:Latin}]\label{summation Gessel}
  \[
  L_{m,n}
  = (-1)^{mn} \sum_{(\nu, \mu) \in \polyGessel_m(n)} (-1)^{|{\vof{\nu}}|} \binom{n}{\nu} \binom{n}{\mu} \prod_{ A \subset [m]\atop A\neq \emptyset} (\tau_A)^{\vof{\nu}_A} \cdot (\vof{\nu}_A)! 
  \]
  where
  \[
  \polyGessel_m(n)=
  \setst{
  (\nu, \mu) \in \NN^{\Par{m}} \times \NN^{\Par{m}}
    }{
    |\mu|=|\nu|=n \text{ and }  \vof{\nu}=\vof{\mu}
  }
  .
  \]
    \end{theorem}  
Gessel proved this formula in the language of graphs in \cite{Gessel:Latin}, but mentioned in \cite{Gessel:enumerative} the possibility of proving it with  multisymmetric functions. We will present this proof with multisymmetric functions in Section \ref{sec: proof Gessel}.


In Gessel's Formula , let us take the $\vof{\nu}$ as the indices of the main summation, instead of the pairs $(\nu, \mu)$. 
We observe that a necessary condition on  $\phi \in \RR^{\substar{m}}$ to be  of the form $\phi=\vof{\nu}$ for some  $n$-multiset  $\nu$ of set partitions is
\[
\forall x \in [m],
\quad
\sum_{A:\atop x\in A} \phi_A = n.
\]
We can state:
\begin{equation}\label{eq:Gessel2}
L_{m, n}
=
(-1)^{mn}
\sum_{\phi \in \polysub_m(n)}
(-1)^{|\phi|}
\prod_{ A \subset [m]\atop A\neq \emptyset} (\tau_A)^{\phi_A} \cdot (\phi_A)! \cdot 
\left(
\sum_{\nu \in \polyfiber_m(\phi)}
\binom{n}{\nu}
\right)^2,
\end{equation}
where
\[
\polysub_m(n)
=
\setst{
\phi \in \NN^{\substar{m}}
}{
\forall x \in [m], \sum_{A: x\in A} \phi_A = n
}
,
  \]
  and
\[
\polyfiber_m(\phi)
=
\setst{
\nu \in \NN^{\Par{m}}
}{
\vof{\nu} = \Phi
}
.
  \]
  

\subsubsection{Gessel's Formula as a symbolic method}

Gessel's Formula affords yet another restatement, as a ``symbolic method''. 
The calculations that follow take place in the free commutative algebra $\FA$ generated by the non-empty subset of $[m]$.

\begin{proposition}
  The following procedure gives the number $L_{m,n}$ of Latin rectangles of format $m\times n$: 
  \begin{enumerate}
\item Let
\[
\symbf_m= \sum_{\pi \vDash [m]} \prod_{A \in \pi} A.
\]    
\item Expand $(\symbf_m)^n$ in the basis of monomials.
\item Square each coefficient.
\item Rescale each variable $A$ by multiplying it by $- (-1)^{\#A} \cdot \tau_A $.  
\item For each subset $A$, apply the linear form replacing each $A^k$ with $k!$.
\end{enumerate}
\end{proposition}


As an example, let us apply the recipe for $m=2$.
Remember that the $L_{2,n}/n!$ are  the derangement  numbers.
We should recover a formula for them.

\begin{enumerate}
\item We have $\symbf_2=\set{1}\set{2}+\set{12}$.  
\item The expansion of $(\symbf_2)^n$ is:
  \[
\sum_{i+j=n}  \binom{n}{i}\set{1}^i\set{2}^i \set{12}^j.
\]
\item After squaring the coefficients, we get:
\[
\sum_{i+j=n} \binom{n}{i}^2 \set{1}^i\set{2}^i \set{12}^j.
\]
\item Here, rescaling just consists in changing $\{12\}$ into $-\{12\}$, giving:
  \[
\sum_{i+j=n} (-1)^j \binom{n}{i}^2 \set{1}^i\set{2}^i \set{12}^j.
\]
\item After applying the functionals $A^k \mapsto k!$, we get
  \[
L_{2,n} = \sum_{i+j=n} (-1)^j \binom{n}{i}^2 (i!)^2 j!.
\]    
\end{enumerate}  
After some trivial simplifications of factorials, we recover:
\[
\frac{L_{2,n}}{n!}
=     \sum_{i=0}^n (-1)^{n-i} (n)_i
= n!  \sum_{j=0}^n \frac{(-1)^j}{j!}.
\]

Applying the recipe for $m=3$ gives  the Bogart-Longyear Formula \eqref{eq:BogartLongyear}. For $m=4$, this yields   the formula in \cite[Theorem 4.2]{Pranesachar:chromatic} specialized with $k=n$ (for $k\neq n$, some ``relaxed Latin rectangles'' are  counted instead, see Section \ref{relaxed}).

\subsection{Generalization to 3D Rook polynomials}

We present here a generalization of Gessel's Formula, giving the  3D rook polynomials.

Given any multiset $\nu$ of set partitions of $[m+1]$, let $\wof{\nu}$ be the restriction of $\vof{\nu}$ to the non-empty subsets of $[m]$. We will deal below with the condition $\wof{\nu} = \wof{\mu}$: note that this is different from the condition $\vof{\nu}=\vof{\mu}$ of Gessel's Formula.  

\begin{theorem}\label{sum rp phi}
  The 3D rook polynomials are given by:
  \[
  \rp{m,n,r}{x}
  = (-1)^n \sum_{(\nu, \mu)\in \polyrook_{m+1}(n, r)} \binom{n}{\nu} \binom{r}{\mu} (-x)^{\|\wof{\nu}\|} \prod_{ A \subset [m]\atop A\neq \emptyset} (\tau_A)^{\wof{\nu}_A} \cdot (\wof{\nu}_A)! 
  \]
  where
  \[
  \polyrook_{m+1}(n,r)
  =
  \setst{
    (\nu, \mu) \in \NN^{\Par{m}} \times \NN^{\Par{m}}
  }{
    |\nu|=n
    \text{ and }
    |\mu|=r
    \text{ and }
    \wof{\nu} =  \wof{\mu}
  }
  .
  \]
\end{theorem}  

The proof will be given in Section \ref{sec: proof Gessel rook}.

Equivalently, using the multisets $\wof{\nu}$ as indices, we get that $\rp{m,n,r}{x}$ is 
\[
(-1)^n \sum_{\phi \in \polysubg_m(n,r)}
    (-x)^{\|\phi\|}
    \left( \sum_{\nu \in \polyfiberg_m(n, \phi)} \binom{n}{\nu}     \right)
    \left( \sum_{\mu \in \polyfiberg_m(r, \phi)} \binom{r}{\mu}     \right)
    \prod_{ A \subset [m]\atop A\neq \emptyset} \tau_A^{\phi_A} \cdot (\phi_A)!
    \]
    where
    \[
    \polysubg_m(n, r)
    =
    \setst{
      \phi \in \NN^{\substar{m}}
    }{
      \forall x \in [m],    \sum_{A \subset [m]:\atop x \in A} \phi_A \le \min(n, r),  
    }
    \]
    and
    \[
    \polyfiberg_{m}(n, \phi)
    =
    \setst{ \nu \in \NN^{\Par{m+1}}
    }{\wof{\nu}=\phi \text{ and } |\nu| = n }
    .
    \]

    Yet another restatement is as a symbolic method.
  \begin{proposition}
    The following procedure produces the rook polynomial $\rp{m,n,r}{x}$:
    \begin{enumerate}      
\item consider
\[
\symbfdh_m = \sum_{\pi \vDash [m+1]} \prod_{A \in \pi\atop \text{s.t. } A \subset [m]} A.
\]
(Note that $\symbfdh_m$ is a dehomogenization of $\symbf_{m+1}$, obtained by substituting all $A$ containing $m+1$ with $1$).
\item Expand $(\symbfdh_m)^n$ and $(\symbfdh_m)^r$ in the basis of monomials.
\item Replace in $(\symbfdh_m)^n$ each coefficient with its product with the corresponding coefficient in $(\symbfdh_m)^r$.
\item Rescale each variable $A$ by multiplying it with $-(-x)^{\#A} \tau_A$.  
\item Apply the linear functionals $A^k \mapsto k!$.
\end{enumerate}
\end{proposition}

As an example, consider the case $m=2$.
\begin{enumerate}
\item We have
  \begin{align*}
\symbfdh_2 &= 1 + \set{1} + \set{2} + \set{1}\set{2} + \set{12} \\  
  &= (1+ \set{1})(1+\set{2}) + \set{12},
\end{align*}
\item Then 
\begin{align*}
(\symbfdh_2)^n &= \sum_j \binom{n}{j}(1+\set{1})^{n-j} (1+\set{2})^{n-j} \set{12}^j\\
&= \sum_{j,k,\ell}
  \binom{n}{j} \binom{n-j}{k}\binom{n-j}{\ell} \set{1}^k \set{2}^\ell \set{12}^j.
\end{align*}
The formula for $(\symbfdh_2)^r$ is similar with $n$ changed into $r$.
\item Replacing in $(\symbfdh_2)^n$ each coefficient with its product with the corresponding coefficient in $(\symbfdh_2)^r$, we get
\[
\sum_{j,k,\ell}
\binom{n}{j} \binom{n-j}{k}\binom{n-j}{\ell}
\binom{r}{j} \binom{r-j}{k}\binom{r-j}{\ell}
\set{1}^k \set{2}^\ell \set{12}^j
\]
\item Rescaling, we get
\[
\sum_{j,k,\ell}
\binom{n}{j} \binom{n-j}{k}\binom{n-j}{\ell}
\binom{r}{j} \binom{r-j}{k}\binom{r-j}{\ell}
x^k \set{1}^k x^\ell \set{2}^\ell (-1)^j x^{2j} \set{12}^j.
\]
\item And, finally, applying the linear functionals, we obtain that
$\rp{2,n,r}{x}$ is equal to
\[
\sum_{j,k,\ell}
\binom{n}{j} \binom{n-j}{k}\binom{n-j}{\ell}
\binom{r}{j} \binom{r-j}{k}\binom{r-j}{\ell}
(-1)^j x^{2j+k+\ell} j! k! \ell!
\]
\end{enumerate}

For $m=3$, we find that $\rp{3,n,r}{x}$ is the sum of the 
\begin{multline*}
  (-1)^{i+j+k}
  \frac{a!b!c!}{i!j!k!\ell!q!}  \frac{2^s}{s!} x^{a+b+c+2i+2j+2k+3s}\\
  \times \binom{\ell+i}{a} \binom{\ell+j}{b} \binom{\ell+k}{c}
  \binom{q+i}{a} \binom{q+j}{b} \binom{q+k}{c}
\end{multline*}
 over all tuples $(\ell, q,i,j,k, s,a,b,c)$ such that $\ell+i+j+k+s=n$ and $q+i+j+k+s=r$.

\subsection{Case of relaxed Latin rectangles and colorings of the rook's graph}\label{relaxed}

Let us call \emph{relaxed Latin rectangles} of format $m\times n\times r$ the arrays with $m$ rows, $n$ columns, filled with $r$ symbols, fulfilling the column condition of Latin rectangles (every symbol appears at most once) and the following modified row condition: ``every symbol appear \emph{at most} once in every row''. Let us denote with $L_{m,n,r}$ the number of such relaxed Latin rectangles. (the case of Latin rectangles is with $r=n$, i.e. $L_{m,n,n}=L_{m,n}$).

Formulas for numbers $L_{m,n,r}$ have been considered in \cite{Pranesachar:chromatic}. As observed there, the number $L_{m,n,r}$ interprets also as the number of proper edge colorings of the  complete bipartite graph $K_{m,n}$, when $k$ colors are available; or, equivalently, the value at $r$ of the chromatic polynomial of the line graph of $K_{m,n}$. This line graph is the  graph whose vertices are the cells of a $m\times n$ array, and  the pairs of adjacent vertices are the pairs of cells in the same row or in the same colors (``rook's graph''). 

Here is the analogue of Gessel's Formula for counting relaxed Latin rectangles.
\[
  L_{m,n,r}
  = (-1)^n \sum_{(\nu, \mu)\in\lozenge_m(n,r)} (-1)^{\|\vof{\nu}\|} \binom{n}{\nu} \binom{r}{\mu} \prod_{ A \subset [m]\atop A\neq \emptyset} \tau_A^{\vof{\nu}_A} \cdot (\vof{\nu}_A)! 
  \]
  where
  \[
  \lozenge_m(n,r)=\setst{
    (\nu,\mu) \in \NN^{\Par{m}} \times \NN^{\Par{m+1}}
  }{
    |\nu|=n,
    |\mu| = m,
    \text{ and }
    \vof{\nu} = \wof{\mu}
  }
  .
  \]
  
See Section \ref{gseries relaxed} for more results on counting relaxed Latin rectangles.

  \subsection{Domains of summation as polytopes}

  Gessel's Formula, as well as its restatements and generalizations to rook polynomials, are summations over  large sets of indices. For Gessel's Formula, as stated in Theorem \ref{summation Gessel}, the number of indices $\nu_\pi$ and $\mu_\pi$ is twice the number of set partitions of $[m]$. One can eliminate some of the indices using the linear relations between them. For instance, for $m=4$, Pranesachar's formula uses only $18$ indices instead of the $30$ indices $\nu_\pi$ and $\nu_\pi$. Yet, the choices of what variables are eliminated is arbitrary, and hides the structure of the set of indices.

  We find useful to do the contrary, and make explicit this structure: the domains of summations $\polyGessel_m(n)$, $\polysub_m(n)$, $\polyfiber_m(\phi)$, $\polyrook_m(n,r)$,  $\polysubg_m(n, r)$, $\polyfiberg_m(n,\phi)$ and $\lozenge_m(n, r)$ are each the \emph{set of lattice points of a polytope} (whose description is obtained from the description we gave of the corresponding set, by replacing $\NN$ with $\RR^+$).

For instance, $\polysub_m(n)$ is the $n$-dilate of $\polysub_m(1)$, which is the Newton polytope of the function $f_m$ of the symbolic method for Gessel's Formula.
  

For another example, remember that the number $L_{m,n}$ of Latin rectangles of format $m\times n$ is, when $n \ge m$,  the leading coefficient of $\rp{m,n,n}{x}$ (degree $mn$). This can now be understood as follows: the polytope $\polysub_m(n)$  in Theorem  \ref{sum rp phi} is the face of the polytope $\polysubg_{m}(n, n)$
defined by the equation:
\[
\sum_{A \subset [m]\atop A \neq \emptyset} \phi_A\cdot \#A = mn.
\]
(Or, equivalently, by the $m$ equations $\sum_{A: x\in A} \phi_A=n$).


\section{Generating series}\label{sec:results}

Gessel's Formula and its generalization involve sums on  sets of indices that become very large. It is natural to ask for more compact formulas. This is done here by considering, instead of explicit formulas for single numbers of Latin rectangles, or rook polynomials, the generating series $g_m$ for these families of numbers.

Here we present our formulas for the generating series $g_2$ and $g_3$. They correspond to the case ofr rook placements on 3D chessboards with 2 or 3 levels.
We also include the case of $g_1$, that, despite being very classic, is interesting and fits well in the general picture. The proofs will be given in Section \ref{proofs g2 g3}.

\subsection{Rook placements on a $1 \times n \times r$ chessboard}

This case corresponds to the ordinary (i.e. 2-dimensional) rook placements on a $n\times r$  rectangular board. 
It is well-known that the corresponding rook polynomials are given by:
\begin{equation}\label{r1mn}
\rp{1,n,r}{x} =  \sum_d (n)_d (r)_d \frac{x^d}{d!},
\end{equation}
where $(n)_k=n(n-1)\cdots (n-(k-1))$ (falling factorial). 
As a direct consequence, 
\begin{equation}\label{g1}
g_1 = e^{z+t+ztx}.
\end{equation}

For later, we take note of the following property of the ordinary rook polynomials, that follows from \eqref{r1mn}:
$\rp{1,n,r}{x} = (r)_n x^{n} + o(x^{n})$. Therefore, when $r \ge n$, the leading term of $\rp{1,n,r}{x}$ is $(r)_n x^n$.

\subsection{Rook placements on a $2\times n \times r$ chessboard}

We get the following expression for $g_2$:
\begin{equation}\label{g2}
g_2 = e^{- ztx^2} \cdot \sum_{n,r}  \left( \rp{1,n,r}{x} \right)^2
\frac{z^n}{n!}\frac{t^r}{r!}. 
\end{equation}

Let $\Hadamard$ stand for the Hadamard product (coefficientwise product) for exponential series in the variables $z$ and $t$.
It is defined as:
\[
\left(\sum_{n,r} a_{n,r} \frac{z^n}{n!} \frac{t^r}{r!}\right)
\Hadamard
\left(\sum_{n,r} b_{n,r} \frac{z^n}{n!} \frac{t^r}{r!}\right)
=
\sum_{n,r} a_{n,r} b_{n,r} \frac{z^n}{n!} \frac{t^r}{r!}.
\]
Then \eqref{g2} can be restated as
\begin{theorem}
\[
g_2 = e^{-ztx^2} \cdot (g_1 \Hadamard g_1).
\]
\end{theorem}
This formula has a clear combinatorial interpretation: the  placements   of $d$ rooks on a $2 \times n \times r$ chessboard  are obtained from the placements of the rooks on each level (two bidimensional rook placements on a $n\times r$ chessboard) by excluding those that have rooks in coinciding position.

\subsection{Rook placements on a  $3\times n \times r$ chessboard}

For $m=3$, we get the following formula:
\begin{equation}\label{g3:F}
  g_3 =
e^{2 ztx^3}  \cdot
\sum_{n,r} \left( F\left(n,r,x,-ztx^2\right)\right)^3
\frac{z^n}{n!}\frac{t^r}{r!} ,
\end{equation}
where
\[
F(n,r,x,y) = \sum_{ a} \rp{1, n+a, r+a}{x} \; \frac{y^a}{a!}.
\]


Formula \eqref{g3:F} can be expressed  as an Hadamard product as well. 
\begin{theorem}
  Set 
\[
\Phi = \sum_{n, r} F(n,r, x, y) \; \frac{z^n}{n!} \frac{t^r}{r!}.
\]
Then
\begin{equation}\label{g3}
g_3 = e^{2 ztx^3}  \cdot \left( \Phi \Hadamard \Phi \Hadamard \Phi\right)_{|y = -ztx^2}.
\end{equation}
\end{theorem}

\subsection{Implementation}

The formulas for $g_1$, $g_2$ and $g_3$ presented above afford remarkably concise implementations. Here is such an implementation in SageMath \cite{Sagemath}, to get their expansion up to order $N$.

\begin{lstlisting}[language=Python]
A.<x> = PolynomialRing(QQ)
B.<z, t> = PowerSeriesRing(A)

def Hadamard(P, Q, N):
    """Hadamard product of two exponential series P and Q
    in the variables z and t, up to order N"""
    return sum(P[i,j]*Q[i, j]*factorial(i)*factorial(j)*z^i*t^j 
               for (i, j) in P.exponents())

def g1(N):
    """Exponential generating series for the rook polynomials
    for a  1 x n x r chessboard,  up to order N"""
    return (z+t+x*z*t).exp(prec=N+1)    

def g2(N):
    """Exponential generating series for the rook polynomials
    for a 2 x n x r chessboard"""
    return Hadamard(g1(N), g1(N), N) * (-x^2*z*t).exp(prec=N+1)

def F(n, r, N):
    return sum(x^i/factorial(i) * (-z*t*x^2)^a/factorial(a)
                   * falling_factorial(n+a, i)
                   * falling_factorial(r+a, i)
                   for a in [0..(N-n-r)/2]
                   for i in [0..min(n,r) +a]) 

def g3(N):
    """Exponential generating series for the rook polynomials
    for a  3 x n x r chessboard"""
    return (sum(F(n,r,N)^3 * z^n/factorial(n) *t^r/factorial(r)
            for (n, r) in IntegerListsLex(length=2, max_sum=N)
             ) * (2*z*t*x^3).exp(prec=N+1))
\end{lstlisting}


\subsection{From 3D rook placements to Latin rectangles}\label{specialization}

Since the rook polynomial $\rp{m,n,n}{x}$ contains as a coefficient (the one of degree $mn$) the number of Latin rectangles of format $m\times n$, it is natural to ask for a process to derive from $g_m$ a generating series for the numbers of Latin rectangles with $m$ lines.

Here is such a process: (1) replace $x$ with $1/x$; (2) replace $z$ with $zx^m$; (3) replace $x$ with $0$; (4) replace $z$ with $z/t$; (5) replace $t$ with 0. This yields the doubly exponential generating series:
\[
\sum_n  L_{m,n} \frac{z^n}{(n!)^2}.
\]
Writing it as
\[
\sum_n \frac{L_{m,n}}{n!} \frac{z^n}{n!},
\]
we interpret the series as the exponential generating series for the numbers $L_{m,n}/n!$   
of  so-called \emph{normalised} Latin rectangles (the Latin rectangles with first line $1,2,\ldots, n$).

The reader can check that, applying this process to $g_1$, $g_2$ and $g_3$, one arrives to:
\[
\sum_n \frac{L_{1,n}}{n!} \frac{z^n}{n!} = e^z,
\]
\[
\sum_n \frac{L_{2,n}}{n!} \frac{z^n}{n!} =
\frac{e^{-z}}{1-z},
\]
\[
\sum_n \frac{L_{3,n}}{n!} \frac{z^n}{n!}
= e^{2z} \sum_n \frac{n! z^n}{(1+z)^{3n+3}}.
\]

The series for $m=1$ and $m=2$ are easily explained. For $m=1$: the 1-line Latin rectangles are just the permutations, whence  $L_{1,n}=n!$. For $m=2$, one recognizes the exponential generating series of the Derangement Numbers. Indeed, the 2-lines normalised Latin rectangles have the derangements as their second line. Finally, the formula for $m=3$ appears in  \cite[Ex.4.5.10 and its solution]{GouldenJackson}; see also \cite{Gessel:3lines}.

\subsection{From 3D rook placements to colorings of the rook's graphs}\label{gseries relaxed}

In Section \ref{relaxed}, we have introduced  the number $L_{m,n,r}$ of relaxed Latin rectangles of format $m\times n$ filled with numbers in $[r]$.
We see that  $L_{m,n,r}$ is the coefficient of degree $mn$ of the rook polynomial $\rp{m,n,r}{x}$, as well as its leading coefficient, provided that $r \ge m$ and $r \ge n$ (if these conditions are not satisfied then $L_{m,n,r}=0$).
Therefore, the generating series
\[
\sum_{n,r} L_{m,n,r} \frac{z^n}{n!} \frac{t^r}{r!}
\]
can be recovered from $g_m$ following steps (1), (2) and (3)  of  section \ref{specialization}, but skipping (4) and (5). We obtain the following formulas:
\begin{corollary}
\[
\sum_{n,r} L_{1,n,r} \frac{z^n}{n!} \frac{t^r}{r!} =
  e^{t+zt},
\]
\[
\sum_{n,r} L_{2,n,r} \frac{z^n}{n!} \frac{t^r}{r!} =
  \frac{e^{-zt}}{1-zt} \exp\left(\frac{t}{1-zt}\right),
\]
\[
\sum_{n,r} L_{3,n,r} \frac{z^n}{n!} \frac{t^r}{r!} =
  e^{2zt} \sum_{n,r} \frac{\left((r)_n\right)^3}{(1+zt)^{3r+3}} \frac{z^n}{n!} \frac{t^r}{r!}.
  \]
\end{corollary}  
The results can be expressed with Hadamard products: 
\[
\sum_{n,r} L_{2,n,r} \frac{z^n}{n!} \frac{t^r}{r!} =
  e^{-zt} \cdot \left( e^{t(1+z)}\Hadamard e^{t(1+z)}\right),
  \]
and   
\[
\sum_{n,r} L_{3,n,r} \frac{z^n}{n!} \frac{t^r}{r!} 
=
e^{2zt} \cdot \left(Q_3 \Hadamard Q_3 \Hadamard Q_3\right)_{| y = zt} 
\]
where
\[
Q_3 = \frac{1}{1+y} \exp\left(\frac{t(1+y+z)}{1+y} \right).
\]

\subsection{Packed rook placements}

Note that each rook placement is counted infinitely many times by the generating series $g_m$. Indeed, for any given any rook placement, one can add empty slices parallel to the planes of coordinates to get the same rook placement on a bigger chessboard.   Let us call \emph{packed rook placement} a rook placement with no empty slice. Let $\frp{m,n,r}{x}$ be the ordinary generating function for packed rook placements on a $m\times n \times r$ chessboard (``packed rook polynomial''). 
  Then the generating series for ordinary rook placements and packed rook placements are related with the simple relation:
\[
\sum_{m, n,r} \rp{m, n, r}{x} \frac{u^m}{m!} \frac{z^n}{n!}\frac{t^r}{r!} 
=
e^{z+t+u} \cdot 
\left( \sum_{m, n,r} \frp{m, n, r}{x} \frac{u^m}{m!} \frac{z^n}{n!}\frac{t^r}{r!}\right).
\]


\section{Multisymmetric functions}\label{sec:multisymmetric}

MacMahon showed that the number of Latin rectangles is a coefficient of the expansion of the power of the permanent. Refining this result and taking profit of symmetries of the permanent,  he also showed that this coefficient extraction could be performed in the ring of multisymmetric functions. Gessel clarified and modernized  MacMahon's work. The approach with multisymmetric functions fits like a glove to the derivation of formulas for the 3D rook polynomials and the use of generating series.

In this section, we present briefly the multisymmetric functions. We review how to derive Gessel's Formula for $L_{m,n}$ with them. Then we show how to extend it for rook polynomials. Finally, we use these calculations with multisymmetric polynomials to derive our formulas for the series $g_2$ and $g_3$.

\subsection{Counting Latin rectangles with polynomials}

To each array with $m$ rows and $n$ columns, filled with symbols from $1$ to $n$ we associate a sequence of $n$ squarefree monomials of degree $m$ in variables $x_{i, j}$, for $i\in [n]$ and $j\in [m]$, as follows: the $k$-th monomial admits $x_{i, j}$ as a factor when the $k$-th column has a symbol $i$ in position $j$.

So, for instance, when $m=2$ and $n=3$,  the sequences of monomials associated to the arrays
\[
\begin{bmatrix}
  1 & 2 & 3 \\
  3 & 1 & 2
\end{bmatrix}
\text{ and }
\begin{bmatrix}
  2 & 2 & 3 \\
  1 & 3 & 2
\end{bmatrix}  
\]
are 
$(x_{11} x_{32} ; x_{21} x_{12} ; x_{31} x_{22})$ and $(x_{21 } x_{12} ; x_{21} x_{32} ; x_{31} x_{22})$.

A column of the array fulfills the column condition of Latin rectangles (no symbol repeated) when its monomial is  of the form
$x_{f(1) 1} x_{f(2) 2} \cdots x_{f(m) m}$ for $f: [m] \hookrightarrow [n]$ injective. The sum of all such monomials is the permanent of the rectangular matrix of the $x_{i,j}$, which is:
\[
\per\left((x_{i,j})_{i\in [n], j\in [m]}\right) = \sum_{f: [m] \hookrightarrow [n] } \prod_{j=1}^m x_{f(j) j}.
\]
The sequences of monomials that fulfill the row conditions (each symbol appears exactly one in each row )  are those whose product is the product of all variables $x_{ij}$ for $i \in [n]$ and $j\in [m]$. 
We deduce, as MacMahon: 
\begin{equation}\label{permanent}
  \left(
  \per\left((x_{i,j})_{i\in [n],  j\in [m]}\right)
  \right)^n = L_{m,n} \cdot \prod_{i\in [n]\atop j\in [m]} x_{i,j} + \text{ combination of other monomials.}
\end{equation}
MacMahon observed that this identity could be recast in the setting of multisymmetric functions, and that the coefficient extraction could be performed in that setting.

\subsection{The algebra of multisymmetric functions}

The $m$-fold \emph{multisymmetric functions} (also known as \emph{vector symmetric functions}) are the formal  series in  variables $x_{i,j}$, for $i\in \NN$ and $j\in [m]$,   that are the entries of a matrix with countably many rows and $m$ columns, invariant under any permutation of finitely many of the rows, and with finitely many homogeneous components (the case $m=1$ corresponds to the classical symmetric functions). 
MacMahon studied in detail these multisymmetric functions \cite[Vol.2, Section XI]{MacMahon:book}.
He used them to get new methods and results in enumerative combinatorics, in particular for counting Latin rectangles.
The results we presented in Section \ref{sec:results} were obtained by applying MacMahon's methods, as explained, updated and extended by Gessel in \cite{Gessel:enumerative}.

Let us give a quick idea of the calculations involved in the MacMahon-Gessel approach. The algebra of $m$-fold multisymmetric functions (with rational coefficients) is equipped  with analogues of classical families of symmetric functions: power sums $p_{\alpha}$, elementary functions $e_{\alpha}$ and complete sums $h_{\alpha}$. 
Each of these families is indexed by the non-zero tuples $\alpha\in \mathbb{N}^m$, and generate freely the algebra. The algebra of $m$-fold multisymmetric functions is graded with a degree with values in $\NN^m$, and each of $p_{\alpha}$, $e_{\alpha}$ and $h_{\alpha}$ is homogeneous of degree $\alpha$.

The multisymmetric power sums are defined by:
\[
p_{\alpha} = \sum_{i=1}^{\infty} \multi{x}_i^{\alpha}
\]
where $\multi{x}_i^\alpha$ stands for $x_{i,1}^{\alpha_1} \cdots x_{i,m}^{\alpha_m}$. 
The complete sums and the elementary functions are best defined by means of generating series. Let $t_1, \ldots, t_m$ be variables, then
\[
\sum_{\alpha\in\NN^m} h_{\alpha} \multi{t}^{\alpha} = \prod_{i=1}^{ \infty} \frac{1}{1-\sum_{j=1}^m x_{i,j} t^j},
\quad
\sum_{\alpha\in\NN^m} e_{\alpha} \multi{t}^{\alpha} = \prod_{i=1}^{ \infty} \left(1+\sum_{j=1}^m x_{i,j} t^j\right),
\]
where we conveniently defined $e_{(0,0,\ldots,0)}$ and $h_{(0,0,\ldots,0)}$ to be $1$.
From these definitions follow the relations:
\[
\sum_{\alpha\in\NN^m} h_{\alpha} \multi{t}^{\alpha}  = \exp\left(
\sum_{\alpha \in \NN^m \atop \alpha\neq \multi{0}} \frac{p_{\alpha}}{Z_{\alpha}} \multi{t}^{\alpha}
\right),
\quad
\sum_{\alpha\in\NN^m} e_{\alpha} \multi{t}^{\alpha}  = \exp\left(
-\sum_{\alpha \in \NN^r \atop \alpha\neq \multi{0}} \frac{p_{\alpha}}{Z_{\alpha}} (-\multi{t})^{\alpha}
\right),\]
where
\[
Z_{\alpha}=\frac{\alpha_1! \alpha_2! \cdots \alpha_m!}{(\alpha_1+\alpha_2+\cdots+\alpha_m-1)!}.
\]

In addition, this algebra is endowed with a scalar product (analogue of Hall's scalar product for symmetric functions) for which 
the monomials in th power sums are orthogonal and fulfill:

\begin{equation}\label{eq:orthogonality}
\scalar{\prod_{\alpha} p_{\alpha}^{k_{\alpha}}}{\prod_{\alpha} p_{\alpha}^{\ell_\alpha}}
= \prod_{\alpha} \scalar{ p_{\alpha}^{k_{\alpha}}}{ p_{\alpha}^{\ell_\alpha}}
=
\begin{cases}\prod_\alpha \left(Z_{\alpha}\right)^k k!
  & \text{ if }\forall \alpha, \;k_\alpha = \ell_\alpha, \\
  0 & \text{otherwise.}
  \end{cases}
\end{equation}

A key feature in our use of multisymmetric functions is that extraction of coefficients of monomials can be performed  by means of scalar products with products of complete sums:
\begin{lemma}[{see \cite{Gessel:enumerative}}]\label{extraction}
  Let $f$ be a multisymmetric function. Let $\vec{u}_1$, $\vec{u_2}$, \ldots, $\vec{u}_r$ be vectors in $\NN^m$.

  The coefficient of $\multi{x}_1^{\vec{u}_1} \multi{x}_2^{\vec{u}_2} \cdots \multi{x}_r^{\vec{u}_r} $ in the expansion in monomials of $f$ is
  \[
  \scalar{f}{ h_{\vec{u}_1} h_{\vec{u}_2} \cdots h_{\vec{u}_r}}.
  \]
\end{lemma}  

About elementary functions,  we will use mainly  those indexed by 0-1 vectors $\chi(A)$ (characteristic vector of a subset $A \subset [m]$) . These are:
\[
e(A) := e_{\chi(A)} = \sum_{f} \prod_{j \in A} x_{f(j)j},
\]
where the sum is over all injections $f:A \hookrightarrow \NN$, for $A \neq \emptyset$. We keep this notation for $A=\emptyset$, so that $e(\emptyset)=1$ (without being an elementary function). In particular,
\begin{multline*}
e([m]) = \per\left((x_{i,j})_{i\in [n], j\in [m]}\right) \\+ \text{monomials involving some variable $x_{i,j}$ with $i>n$.}
\end{multline*}

We see that \eqref{permanent} restates, in the setting of multisymmetric functions, as: $L_{m,n}$ is the coefficient of $\multi{x}_1^{(1,1,\ldots,1)} \multi{x}_2^{(1,1,\ldots, 1)} \cdots \multi{x}_n^{(1,1,\ldots, 1)}$ in the expansion in monomials of $e([m])^n$. We deduce, after Lemma \ref{extraction}, that

\begin{theorem}[\cite{MacMahon:book, Gessel:enumerative}]\label{Latin scalar}
\[
L_{m,n} = \scalar{e([m])^n}{h([m])^n}
\]
\end{theorem}

\subsection{The subalgebra of subsets}

Many of the calculations that follow take place in a subalgebra of the algebra of multisymmetric function isomorphic to the free commutative algebra $\FA$ generated by the non-empty subsets of $[m]$.
More precisely, 
we embed $\FA$ into  the algebra of multisymmetric functions by sending $A$ to $p_{\chi(A)}$, where $\chi(A)$ is the characteristic vector of $A$. 
So, for instance, for $m=3$, the sets
$
\set{1}, \set{2}, \set{13}, \set{123}
$, seen as elements of $\mathbb{A}_3$,   are mapped respectively to
$
  p_{100}, p_{010}, p_{10,1} \text{ and } p_{111}. 
$
  From now on, we use $A$ for $p_{\chi(A)}$.

  The scalar product of multisymmetric functions restricts to $\FA$. The monomials in the variables $A$ are orthogonal, and 
  \[
  \scalar{\prod_{A \subset [m]\atop A \neq \emptyset} A^{k_A}}{\prod_{A \subset [m]\atop A \neq \emptyset} A^{k_A}}=
  \prod_{A \subset [m]\atop A \neq \emptyset} \left(\frac{1}{\tau_A}\right)^{k_A} \cdot k_A!
    \]
 since $1/Z_{\chi(A)}=(\#A - 1)!=\tau(A)$.

It is often lighter to use a normalized version of the variables $A$.
Set $\norma{A}=\tau_A \cdot A$. 
Then
  \[
  \scalar{\prod_{A \subset [m]\atop A \neq \emptyset} (\norma{A})^{k_A}}{\prod_{A \subset [m]\atop A \neq \emptyset} (\norma{A})^{k_A}}=
  \prod_{A \subset [m]\atop A \neq \emptyset} \left({\tau_A}\right)^{k_A} \cdot k_A!
    \]
  
\subsection{Proof of Gessel's Formula}\label{sec: proof Gessel}

To prove Gessel's Formula, we start with $L_{m,n}= \scalar{(e([m]))^n}{(h([m]))^r}$ (Theorem \ref{Latin scalar}).
The expansions in power sums  of the functions $e([m])$ and $h([m])$ are known, they were given by Mercedes Rosas.

\begin{theorem}[\cite{Rosas}]\label{Rosas}
  For any subset $B$ of $[m]$,  
  \[
  e(B) = (-1)^{\#B}
  \sum_{\pi \vDash B} \prod_{A\in\pi}   (-\norma{A}),
  \quad
  \text{ and }
  \quad
  h(B) =
  \sum_{\pi \vDash B} \prod_{A\in\pi}  \norma{A}.
  \]
  \end{theorem}
(In particular, the $e(B)$ and the $h(B)$ belong to $\FA$.)

  As a consequence,
  \[
   e([m])^n
  = (-1)^{mn}\sum_\nu \binom{n}{\nu} \prod_{\pi\vDash[m]} \prod_{A \in \pi} (-\norma{A})^{\nu_\pi}\\
  \]
    where the sum is over all $n$-multisets $\nu$ of set partitions   of $[m]$. Simplifying,  we get
  \begin{align*}
   e([m])^n
   &= (-1)^{mn} \sum_\nu \binom{n}{\nu}  \prod_{A \subset[m]\atop A\neq \emptyset}  (-\norma{A})^{\vof{\nu}_A}
   .
  \end{align*}
Similarly, 
  \[
  h([m])^n 
  = \sum_\mu \binom{n}{\mu}  \prod_{A \subset[m]\atop A\neq \emptyset}  (\norma{A})^{\vof{\mu}_A}
  .
\]
Now Gessel's Formula follows, using the orthogonality of the monomials.

  \subsection{Extension of Gessel's Formula to partial Latin rectangles}\label{sec: proof Gessel rook}

  We now prove Formula \eqref{sum rp phi}. For this, we simply indicate the modifications that need to be brought to the previous derivation of Gessel's Formula .

  Given a $m\times n$  array, partially filled with numbers in $[r]$, we associate to each of its columns the squarefree monomial with a factor $x_{i,j}$ whenever there is an $i$ in position $j$.
  A column fulfills the   column condition for PLRs (``each symbol appears at most once in each column'') when its monomial is of the form $
  \prod_{j \in A} x_{f(j) j}
  $
    for some $A \subset [m]$ and $f: A \hookrightarrow \NN$ injective. The sum of all such monomials is $ \EE_m:= \sum_{A \subset [m]} e(A)$.

    An array fulfills the row condition for PLRs (``each symbol appears at most once in each row'') if and only the product of its monomials  is squarefree. Finally, the numbers used for filling the arry are all in $[r]$ if and only if the product of the monomials involves only variables $x_{i,j}$ where $j\le r$. The number of non-empty entries in the array is the degree of this product.

    Therefore,
    \[
    \rp{m,n,r}{x} = \killer_r  \left((\EE_m)^n\right)
    \]
    where $\killer_r$ is the operator on series in the $x_{ij}$ that replaces each squarefree monomial of degree $d$ in the $x_{i,j}$ for $j\in [r]$ with $x^d$, and kills all other monomials.

    \begin{lemma}
      Let
      \[
      \HH_m = \sum_{A\subset [m]} h(A) \cdot x^{\#A}.
      \]
      For any multisymmetric function $f$, we have
      $
      \killer_r(f) = \scalar{f}{(\HH_m)^r}
      $.
    \end{lemma}  
    \begin{proof}
      This follows from Lemma \ref{extraction}.
    \end{proof}  

    We can conclude:
    \begin{theorem}\label{rook scalar}
    \[
    \rp{m,n,r}{x}
    = \scalar{(\EE_m)^n}{(\HH_m)^r}.
    \]
    \end{theorem}

    After Theorem \ref{Rosas}:
    \[
     \EE_m =  \sum_{A \subset [m]} \sum_{\rho \vDash A} \prod_{B \in \rho} (-1)^{\#B} \cdot (-\norma{B}).
    \]
    But the pairs $(A,\rho)$ such that $A\subset [m]$ and $\rho \vDash A$ are in bijection with the partitions $\pi \vDash [m+1]$:  if $\pi=\{B_0, B_1, \ldots, B_k\}$ where $B_0$ is the block  containing $m+1$, the corresponding pairs $(A, \rho)$ is with $A=[m+1] \setminus B_0$ and $\rho=\{B_1, \ldots B_k\}$.

    So
    \[
    \EE_m =  \sum_{\pi \vDash [m+1]} \prod_{B \in \pi\atop
    \text{s.t.} m+1 \not \in B} (-1)^{\#B}  \cdot (-\norma{B}).
    \]

    Similarly,
    \[
    \HH_m = \sum_{\pi \vDash [m+1]} \prod_{B \in \pi\atop
    \text{s.t.} m+1 \not \in B}  x^{\# B} \cdot \norma{B}.
    \]
    Now, Formula \eqref{sum rp phi} is derived by expanding the powers in the scalar product.

    \subsection{Proofs of the formulas for $g_2$ and $g_3$}\label{proofs g2 g3}

    \subsubsection{The series $g_m$ as a scalar product}
    
    We now turn our attention to the generating series.

    \begin{proposition}
      For all $m>0$, $ g_m = \scalar{\exp(z\EE_m)}{\exp(t\HH_m)}$.
    \end{proposition}
    \begin{proof}
    Start with $\rp{m,n,r}{x} = \scalar{(\EE_m)^n}{(\HH_m)^r}$.
    Multiply with  $\frac{z^n}{n!} \frac{t^r}{r!}$ and sum over all $n$ and $r$.
\end{proof}    

In the calculations that follow, $m$ is fixed and we drop the indices for $\EE_m$ and $\HH_m$, writing for them just $\EE$ and $\HH$. 
    
\subsubsection{Formula for $g_1$}

Let us consider, as a warm-up,  the formula for $g_1$.

In this case,
$\EE = 1 + \set{1}$ and $\HH = 1 + x \set{1}$.
We have:
\begin{equation}\label{exp and g1}
  g_1 = \scalar{  \exp\left(z(1+\set{1})\right)     }{      \exp\left(t(1+x \set{1})\right) }.
\end{equation}

\begin{lemma}\label{rule exp}
  Let $A \subset [m]$, non-empty, and let  $\alpha$ and $\beta$ be scalars. Then
  \[
  \scalar{e^{\alpha A}}{e^{\beta A}} = \exp\left(\frac{\alpha \beta}{\tau_A}\right).
  \]
\end{lemma}  
\begin{proof}
  Immediate, expanding the exponentials and using the orthogonality of the $A^k$.  
\end{proof}  

Starting from \eqref{exp and g1},
we have:
\[  g_1
  =
  \scalar{e^z e^{z\set{1}}}{e^t e^{t x \set{1}}}
  = e^z e^t \scalar{e^{z\set{1}}}{e^{t x \set{1}}}.
\]
Applying Lemma \ref{rule exp}, we deduce 
\[ g_1
  =
  e^z e^t e^{ztx},
  \]
  which is Formula \eqref{g1}.


We take note, in view of a later use, that, as a consequence of \eqref{exp and g1},
\begin{equation}\label{R1}
  \rp{1,n,r}{x} = \scalar{(1+\set{1})^n}{(1+x\set{1})^r}.
\end{equation}


\subsubsection{Formula for $g_2$}\label{sec: proof g2}

We start with stating two rules for separating variables. 
\begin{lemma}\label{separation 1}
  Let $X$ and $Y$ be two disjoint sets of power sums. Let $P_1$ and  $P_2$ be polynomials in the power sums in $X$, and let $Q_1$ and $Q_2$ be polynomials in the power sums in $Y$. Then:
  \[
   \scalar{P_1 Q_1}{P_2 Q_2}
  =
  \scalar{P_1}{P_2}
  \cdot
  \scalar{Q_1}{Q_2}.
  \]
\end{lemma}  
\begin{proof}
  Immediate.
\end{proof}  

\begin{lemma}\label{separation}
  Let $X$ and $Y$ be two disjoint sets of power sums. Let $F_1, F_2, P_1$ and  $P_2$ be polynomials in the power sums in $X$, and let $G_1, G_2, Q_1$ and $Q_2$ be polynomials in the power sums in $Y$, each independent from $z$ and $t$, we have
  \[
  \scalar{P_1 Q_1 \cdot e^{z F_1 G_1}}{P_2 Q_2 \cdot e^{t F_2 G_2}}
  =
  \scalar{P_1 \cdot e^{z F_1}}{P_2 \cdot e^{t F_2}}
  \Hadamard
  \scalar{Q_1 \cdot e^{z G_1}}{Q_ 2 \cdot e^{t G_2}}.
  \]
\end{lemma}
\begin{proof}
  Immediate, expanding the exponentials.
\end{proof}  

We now prove the formula for $g_2$.
We have
$
g_2=\scalar{e^{z\EE}}{e^{t\HH}}
$,
 with
\[
\EE = (1+ \set{1}) (1+ \set{2}) - \set{12},
\qquad
\HH = (1+ x\set{1}) (1+ x\set{2}) + x^2\set{12}.
\]
Applying Lemma \ref{separation 1}, we see that
\[
g_2 =
\scalar{e^{(z (1+ \set{1}) (1+ \set{2})}}{e^{t  (1+ x\set{1}) (1+ x\set{2}}}
\cdot 
\scalar{e^{-z \set{12})}}{e^{tx^2 \set{12}}}.
\]
For the second factor, we apply Lemma \ref{rule exp}, and get 
\[
\scalar{e^{-z \set{12}}}{e^{tx^2 \set{12}}} = \exp(-ztx^2).
\]
The first factor is evaluated making use of Lemma \ref{separation}:
\begin{multline*}
\scalar{e^{z (1+ \set{1}) (1+ \set{2})}}{e^{t  (1+ x\set{1}) (1+ x\set{2})}}\\
=
\scalar{e^{z (1+ \set{1})}} {e^{t  (1+ x\set{1})}}
\Hadamard
\scalar{e^{z (1+ \set{2})}} {e^{t  (1+ x\set{2})}}.
\end{multline*}
After \eqref{exp and g1},
\[
\scalar{e^{z (1+ \set{1})}} {e^{t  (1+ x\set{1})}} = g_1,
\]
and the same holds for $\scalar{e^{z (1+ \set{2})}} {e^{t  (1+ x\set{2})}}$.
Whence,
\[
g_2 =\left(g_1 \Hadamard g_1\right) \cdot  \exp(-ztx^2) .
\]


\subsubsection{Formula for $g_3$}\label{sec: proof g3}

We have $g_3 = \scalar{e^{z\EE}}{e^{t\HH}}$, with
\begin{multline*}
\EE = 
(1+ \set{1}) (1+ \set{2})  (1+ \set{3}) \\
- (1+\set{1})\set{23}
- (1+\set{2})\set{13}
- (1+\set{3})\set{12}
+2 \set{123}
,
\end{multline*}
and
\begin{multline*}
\HH = (1+ x\set{1}) (1+ x\set{2})  (1+ x\set{3})\\
+x^2 (1+x\set{1})\set{23}
+x^2 (1+x\set{2})\set{13}
+x^2 (1+x\set{3})\set{12}
+2 x^3 \set{123}.
\end{multline*}
Separating the variable $\set{123}$ from the others,
we decompose $\EE$ and $\HH$ into
\[
\EE = \EE'  +2 \set{123},
\qquad
\HH = \HH'  +2 x^3 \set{123} 
\]
to  get that
\[
g_3 = \scalar{e^{z\EE' }}{e^{t\HH' }} \scalar{e^{-2z \set{123}}}{e^{2tx^3 \set{123}}}.
\]
We apply Lemma \ref{rule exp}. Since $|\tau_{123}|=2$, 
we get:
\[
\scalar{e^{-2z \set{123}}}{e^{2tx^3 \set{123}}}
=
e^{-2ztx^3}.
\]
To simplify the other factor, consider
\[
\EE'' = (1+ \set{1}) (1+ \set{2})  (1+ \set{3}),
\HH'' =  (1+ x\set{1}) (1+ x\set{2})  (1+ x\set{3})
,
\]
and observe that $\EE' $ and $\HH' $ are obtained from $\EE'' $ and $\HH''$ by adding to them
$ - (1+\set{i})\set{jk}$  and $x^2 (1+x\set{i})\set{jk}$ respectively,
first with $\set{i}=\set{1}$ and $\set{jk}=\set{23}$,
then again with
$\set{i}=\set{2}$ and $\set{jk}=\set{13}$,
then once more with
$\set{i}=\set{3}$ and $\set{jk}=\set{12}$.

We will make use of the following rule to deal with these operations.
\begin{lemma}\label{rule 3}
  Let $i, j, k$ be three distinct elements of $[m]$, and let $F$ and $G$ be elements of $\FA$, both free from the variable $\set{ij}$. Then:
  \[
  \scalar{F e^{-z(1+\set{k}) \set{ij}}}{G e^{tx^2(1+x\set{k})\set{ij }}}
    =
    \sum_a \frac{y^a}{a!} \scalar{F \cdot (1+\set{k})^a}{G\cdot (1+x \set{k})^a},
  \]
  where $y=-ztx^2$.
\end{lemma}
We apply  the rule three times, and get
\[
g_3 = e^{-2ztx^3} \cdot 
\sum_{a,b,c} \frac{y^{a+b+c}}{a!b!c!} X_{a,b,c}
,
\]
  where $X_{a,b,c}$ is the scalar product of
\[
(1+\set{1})^a(1+\set{2})^b(1+\set{3})^c \cdot e^{z \EE'' }
\]
with
\[
(1+x\set{1})^a(1+x\set{2})^b(1+x\set{3})^c \cdot e^{t \HH''}
.
\]
Applying Lemma \ref{separation}, we get
\[
X_{a, b, c} = \phi_a \Hadamard \phi_b \Hadamard \phi_c
,
\]
where
\[
\phi_k = \scalar{(1+\set{1})^k e^{z(1+\set{1})}}{(1+x\set{1})^k e^{t(1+x\set{1})}}
.
\]
Therefore
\[
g_3 = e^{-2ztx^3} \cdot  
\left( \sum_c \phi_c \frac{y^c}{c!}\right)
\Hadamard
\left( \sum_b \phi_a \frac{y^b}{b!}\right)
\Hadamard
\left( \sum_a \phi_a \frac{y^a}{a!}\right)
.
\]
Expanding the exponentials in $\phi_k$ yields:
\[
\phi_k = \sum_{n, r} \frac{z^n}{n!} \frac{t^r}{r!} \scalar{(1+\set{1})^{k+n}}{(1+x\set{1})^{k+r}}
.
\]
Comparing with Formula   \eqref{R1}, that says 
\[
\scalar{(1+\set{1})^i}{(1+x\set{1})^j}
=\rp{1,i,j}{x},
\]
we discover that
\[
\phi_k
=
\sum_{n, r} \frac{z^n}{n!} \frac{t^r}{r!}
\rp{1,k+n,k+r}{x}
.
\]
This yields Formula \eqref{g3}.


\section{Final remarks}

\finalremark The complexity of calculating the series  $g_m$  seems to  increase notably with $m$. It would be very interesting  to check what happens for $g_4$, and whether or not one arrives to  short expressions, as it is the case for $g_2$ and $g_3$. 

\finalremark After stating that $L_{m,n}$ can be obtained by extracting one coefficient from    power of the permanent, MacMahon simplified the calculation by considering multisymmetric functions, i.e. by making use of the symmetries with respect to the row permutations of the matrix of the $x_{i,j}$. The other symmetry, with respect to the column permutations, has not been exploited, and has rarely been evoked (one exception is \cite{Jucys}). Can better formulas  be obtained by considering the multisymmetric functions that are invariant under  this other symmetric group action?  

\finalremark The 2D rook polynomials have a well-studied $q$-deformation ($q$-rook polynomials). Could considering such $q$-deformations for partial Latin rectangles  contribute to a better understanding of some divisibility properties of the numbers of Latin rectangles?

\finalremark  \label{rem:Doyle} \emph{Doyle's Formulas} \cite{Doyle,DeGennaro,StonesComputing} are other formulas for counting Latin rectangles. They are actually the best known formulas for effective calculations of $L_{m,n}$ for $m$ and $n$ as big as possible \cite{StonesComputing}. These  formulas  do not generalize to PLRs. Indeed, they are based on  restating the row conditions for Latin rectangles (``every symbol appears \emph{exactly} once in every row'') as
 ``every symbol appears \emph{at least} once in every row'', 
which is equivalent to it in the context of Latin rectangles, that have as many places to fill in as symbols. This is, of course, no longer the case for PLRs.
See Appendix \ref{sec:Doyle} for a derivation of Doyle's Formulas in the setting of multisymmetric functions.

\finalremark Our work does not address the fundamental weakness of Gessel's and Doyle's Formulas (and of our generalizations): these are summations over very large sets, which, when $m$ increases, become bigger than the number of Latin rectangles, so that carrying over the summation becomes more complicated than just listing  the Latin rectangles \cite{Wilf, Zeilberger}. One can check that some summands are much bigger than the final result, but cancel out with summands of opposite signs.
Would it be possible to get simplified  formulas with the cancellations already performed? The presentation of the domain of summations as sets of the lattice points of polytopes might  help. 

\printbibliography

@unpublished{Notebook,
  author = {Briand, Emmanuel},
  title  = {Generating series for  3{D} rook placements on a  $m \times n \times r$ chessboard for small $m$},
  note = {Sagemath Notebook},
  year = {2026},
  url = {https://github.com/EmmanuelJeanBriand/3DRookPlacements/blob/main/3D_rook_placements.ipynb}
}

@InProceedings{Briand:extendedabstract,
  author = 	 {Briand, Emmanuel},
  title = 	 {Tridimensional rook placements and partial {L}atin rectangles through multisymmetric functions},
  date = 	 {2026},
  eventtitle = {Discrete Mathematics Days 2026},
  note = 	 {To appear},
  series = {Research Perspectives},
  publisher = {Birkhauser},
}

@manual{Sagemath,
    label        = {Sage},
    author       = {{The Sage Developers}},
    title        = {{S}age{M}ath, the {S}age {M}athematics {S}oftware {S}ystem},
    url          = {https://www.sagemath.org},
    version      = {10.7},
    year         = {2026},
}

@article{FalconStones2015,
title = {Classifying partial Latin rectangles},
author = {Falcón, Raúl M. and Stones, Rebecca J.},
journal = {Electronic Notes in Discrete Mathematics},
volume = {49},
pages = {765-771},
year = {2015},
note = {The Eight European Conference on Combinatorics, Graph Theory and Applications, EuroComb 2015},
doi = {10.1016/j.endm.2015.06.103},
}

@article {FalconStones2020,
    AUTHOR = {Falc{\'o}n, Ra{\'u}l M. and Stones, Rebecca J.},
     TITLE = {Enumerating partial {L}atin rectangles},
   JOURNAL = {Electron. J. Combin.},
  FJOURNAL = {Electronic Journal of Combinatorics},
    VOLUME = {27},
      YEAR = {2020},
    NUMBER = {2},
     PAGES = {Paper No. 2.47, 41},
       DOI = {10.37236/9093},
}

@article{Gessel:enumerative,
  author = 	 {Gessel, Ira},
  title = 	 {Enumerative Applications of symmetric functions},
  date = 	 {1987},
  journal = {S{\'e}minaire Lotharingien de Combinatoire},
  volume = 	 {17},
  pages = 	 {5-21},
  url = 	 {https://www.mat.univie.ac.at/~slc/opapers/s17gessel.html},
}

@book{MacMahon:book,
 author = {MacMahon, Percy A.},
 title = {Combinatory analysis. {Vols}. {I}, {II} (bound in one volume)},
 edition = {Reprint of {An} introduction to combinatory analysis (1920) and {Combinatory} analysis. {Vol}. {I}, {II} (1915, 1916)},
 year = {2004},
 publisher = {Mineola, NY: Dover Publications},
 language = {English},
}

@article {Rosas,
    AUTHOR = {Rosas, Mercedes H.},
     TITLE = {Mac{M}ahon symmetric functions, the partition lattice, and
              {Y}oung subgroups},
   JOURNAL = {J. Combin. Theory Ser. A},
  FJOURNAL = {Journal of Combinatorial Theory. Series A},
    VOLUME = {96},
      YEAR = {2001},
    NUMBER = {2},
     PAGES = {326--340},
       DOI = {10.1006/jcta.2001.3186},
}

@book {Ryser,
    AUTHOR = {Ryser, Herbert John},
     TITLE = {Combinatorial mathematics},
    SERIES = {The Carus Mathematical Monographs},
    VOLUME = {14},
 PUBLISHER = {Mathematical Association of America, ; distributed by John
              Wiley and Sons, Inc., New York},
      YEAR = {1963},
 }

@PhdThesis{StonesThesis,
  author = 	 {Stones, Douglas S.},
  title = 	 {On the number of Latin rectangles},
  school = 	 {Monash University},
  year = 	 {2009},
  url = {https://bridges.monash.edu/articles/thesis/On_the_number_of_Latin_rectangles/5044384},
}

@article{StonesComputing,
  title = {On Computing the Number of Latin Rectangles},
  author = {Stones, Rebecca J. and Lin, Sheng and Liu, Xiaoguang and Wang, Gang},
  journal = {Graphs and Combinatorics},
  year = {2015},
  doi = {10.1007/s00373-015-1643-1}
}

@misc{Doyle,
      title={The number of Latin rectangles}, 
      author={Doyle, Peter G.},
      year={2007},
      eprint={math/0703896},
      archivePrefix={arXiv},
      doi = {10.48550/arXiv.math/0703896}
}

@article{DeGennaro,
  title = {How {M}any {L}atin {R}ectangles {A}re {T}here?},
  author = {de Gennaro, Aurelio},
  eprint = {arXiv:0711.0527},
  archivePrefix = {arXiv},
  year = {2007},
  doi = {10.48550/arXiv.0711.0527}, 	
}

@article {BogartLongyear,
    AUTHOR = {Bogart, K. P. and Longyear, J. Q.},
     TITLE = {Counting 3 by n {L}atin rectangles},
   JOURNAL = {Proc. Amer. Math. Soc.},
  FJOURNAL = {Proceedings of the American Mathematical Society},
    VOLUME = {54},
      YEAR = {1976},
     PAGES = {463--467},
       DOI = {10.2307/2040843},
}

@article{Pranesachar23,
  author  = {Pranesachar, C. R.},
  title   = {On the number of two-line and three-line latin rectangles, an alternative approach},
  journal = {Discrete Mathematics},
  year    = {1982},
  volume  = {38},
  number  = {1},
  pages   = {79--86},
  doi     = {10.1016/0012-365X(82)90167-1}
}

@article{Pranesachar:chromatic,
  author  = {Athreya, K. B. and Pranesachar, C. R. and Singhi, N. M.},
  title   = {On the Number of {L}atin Rectangles and Chromatic Polynomial of $L(K_{r,s})$},
  journal = {European Journal of Combinatorics},
  year    = {1980},
  volume  = {1},
  pages   = {9--17},
  doi     = {10.1016/s0195-6698(80)80015-1}
}

@article{Riordan,
  author  = {Riordan, John},
  title   = {A Recurrence Relation for Three-Line {L}atin Rectangles},
  journal = {The American Mathematical Monthly},
  year    = {1952},
  volume  = {59},
  number  = {3},
  pages   = {159--162},
  doi     = {10.2307/2306511}
}

@article{Yamamoto49,
  author  = {Yamamoto, Koichi},
  title   = {Asymptotic number of {L}atin rectangles and the symbolic method},
  journal = {S{\^u}gaku},
  year    = {1949},
  volume  = {2},
  number  = {2},
  pages   = {126--128},
  publisher = {Mathematical Society of Japan}
}

@article {Gessel:Latin,
    AUTHOR = {Gessel, Ira M.},
     TITLE = {Counting {L}atin rectangles},
   JOURNAL = {Bull. Amer. Math. Soc. (N.S.)},
  FJOURNAL = {American Mathematical Society. Bulletin. New Series},
    VOLUME = {16},
      YEAR = {1987},
    NUMBER = {1},
     PAGES = {79--82},
       DOI = {10.1090/S0273-0979-1987-15465-6},
}

@InProceedings{Gessel:3lines,
    author = "Gessel, Ira M.",
    editor = "Labelle, Gilbert and Leroux, Pierre",
     title = "Counting three-line {L}atin rectangles",
 booktitle = "Combinatoire {\'e}num{\'e}rative",
eventtitle = {Colloque de Combinatoire {\'E}numerative},
\\location = {Universit{\'e} du Qu{\'e}bec {\`a} Montr{\'e}al},
 eventdate = {May 1985},
    series =  {Lecture Notes in Mathematics},
    volume = 1234,
      year = "1986",
 publisher = "Springer",
     pages = "106--111",
       doi = {10.1007/BFb0072512},
}

@book {GouldenJackson,
    AUTHOR = {Goulden, Ian P. and Jackson, David M.},
     TITLE = {Combinatorial enumeration},
    SERIES = {Wiley-Interscience Series in Discrete Mathematics},
 PUBLISHER = {John Wiley \& Sons, Inc., New York},
      YEAR = {1983},
}

@book {Stanley,
    AUTHOR = {Stanley, Richard P.},
     TITLE = {Enumerative combinatorics. {V}olume 1},
    SERIES = {Cambridge Studies in Advanced Mathematics},
    VOLUME = {49},
   EDITION = {Second Edition},
 PUBLISHER = {Cambridge University Press, Cambridge},
      YEAR = {2012},
}

@article {Other3DRooks,
    AUTHOR = {Alayont, Feryal and Krzywonos, Nicholas},
     TITLE = {Rook polynomials in three and higher dimensions},
   JOURNAL = {Involve},
  FJOURNAL = {Involve. A Journal of Mathematics},
    VOLUME = {6},
      YEAR = {2013},
    NUMBER = {1},
     PAGES = {35--52},
       DOI = {10.2140/involve.2013.6.35},
}

@article{IversonNotation,
 author = {Donald E. Knuth},
 title = {Two Notes on Notation},
 journal = {The American Mathematical Monthly},
 number = {5},
 volume = {99},
 year = {1992},
 pages = {403--422},
 publisher = {Taylor & Francis, Ltd., Mathematical Association of America},
 doi = {10.2307/2325085},
}

@article {Wilf,
    AUTHOR = {Wilf, Herbert S.},
     TITLE = {What is an answer?},
   JOURNAL = {Amer. Math. Monthly},
  FJOURNAL = {American Mathematical Monthly},
    VOLUME = {89},
      YEAR = {1982},
    NUMBER = {5},
     PAGES = {289--292},
         DOI = {10.2307/2321713},
  }

@inbook{Zeilberger,
author = {Zeilberger, Doron},
title = {Enumerative and algebraic combinatorics},
 publisher = {Princeton University Press},
 booktitle = {The Princeton Companion to Mathematics},
 year = {2008},
doi = { 10.2307/j.ctt7sd01},
pages = {550–-561},
}

@article {Jucys,
    AUTHOR = {Jucys, A.-A. A.},
     TITLE = {The number of distinct {L}atin squares as a group-theoretical
              constant},
   JOURNAL = {J. Combinatorial Theory Ser. A},
  FJOURNAL = {Journal of Combinatorial Theory. Series A},
    VOLUME = {20},
      YEAR = {1976},
    NUMBER = {3},
     PAGES = {265--272},
        DOI = {10.1016/0097-3165(76)90020-0},
 }


\clearpage

\appendix

\section{Doyle's  Formulas}\label{sec:Doyle}

In this appendix, we derive Doyle's Formulas in the framework of multisymmetric functions. These derivations are quite natural to one that is acquainted to the calculations with symmetric functions, their scalar product and the adjoints of multiplication operators. At the same time, interesting variants of classical objects show up; in particular multisymmetric analogues of the well-known operators $\sigma^{\perp}$ from classical symmetric function theory, whose effect is to ``add 1 to an alphabet''.

In order to avoid overloading the sum symbols in the formulas, we will make use 
of ``Iverson's notation'' (see \cite{IversonNotation}): given a condition $X(\alpha)$ depending on  the indices $\alpha$ of a summation, $\iverson{X(\alpha)}$ is $1$ if the condition fulfilled, and $0$ otherwise. So, for instance, for $B\subset[m]$, 
\[
\sum_{\alpha: \atop S(\alpha) \subset B} h_\alpha
= \sum_{\alpha} \iverson{S(\alpha) \subset B} \cdot h_\alpha,
\]
where $S(\alpha)$ stands for the \emph{support of $\alpha$}, i.e. the set of all $i$ such that $\alpha_i \neq 0$.

\subsection{Statement of Doyle's Formulas}

Given a multiset $\nu$ of  subsets of $[m]$, let $\nu^+$ be the multiset defined by
\[
\nu^+_A = \sum_{B\atop A\subset B} \nu_B,
\]
let $\|\nu\|=\sum_A \nu_A \cdot \# A $, and for  $f\in \FA$, let   
$
f_{| A \leftarrow \nu_A}
$
be the number obtained by replacing each variable $A$ (for $A\subset [m]$ non-empty) with $\nu_A$. This amounts to replacing each $\norma{A}$ with $\tau_A \cdot \nu_A$.

\begin{theorem}[Doyle's First Formula{\cite[Section 10]{Doyle}} and \cite{DeGennaro}]
We have
    \[
  L_{m,n}
  = (-1)^{mn} \sum_{\nu} (-1)^{\| \nu\|} \binom{n}{\nu} \left( e([m])_{| A \leftarrow \nu^+_A}\right)^n,
    \]
    where the sum is over all $n$-multisets $\nu$  of subsets of $[m]$. 
\end{theorem}

For instance, for $m=1$, and writing $\nu_0$ and $\nu_1$ instead of $\nu_\emptyset$ and $\nu_{\set{1}}$, Doyle's First Formula is
\[
n! = L_{1,n}  = \sum_{\nu_0 +\nu_1 = n} (-1)^{\nu_1} \binom{n}{\nu_0, \nu_1} \left(\nu_1\right)^n.
  \]
Let us rewrite the formula with $k$ for $\nu_1$,  and $n-k$ for $\nu_0$. We get a more familiar formula:
\[
n! = L_{1,n} = (-1)^n \sum_{k=0}^n (-1)^k \binom{n}{k} k^n.
\]

For $m=2$, (writing $\nu_0$ for $\nu_\emptyset$, $\nu_i$ for $\nu_{\set{i}}$ and $\nu_{12}$ for $\nu_{\set{1,2}}$) the formula is:
  \[
  L_{2,n}= \sum_{\nu}
  (-1)^{\nu_1+\nu_2} \binom{n}{\nu_0, \nu_1, \nu_2, \nu_{12}}
      ((\nu_1+\nu_{12})(\nu_2+\nu_{12})-\nu_{12})^n
  \]
      where the sum is over all tuples of nonnegative integers
    $
    \nu=(\nu_0, \nu_1, \nu_2, \nu_{12})
    $
    with sum $n$.

Let us present now Doyle's second formula (which, actually, in Doyle's paper \cite{Doyle} comes first).  It is more complicated to state but more efficient for practical calculations since it is  a sum over a much smaller domain.
\begin{theorem}[Doyle's Second Formula{\cite[Section 8]{Doyle}}, {\cite[Theorem 2]{StonesComputing}}]
    \[
  \frac{L_{m,n}}{n!}
  = (-1)^{mn-n}
  \sum_{\nu} (-1)^{\| \nu\|}
  \binom{n}{\nu}
  \;
  \prod_{B \subset [m-1]} \left(e([m-1])_{|A \leftarrow \nu^+_A-\iverson{A\subset B}}\right)^{\nu_B}
  \]
    where the sum is over all $n$-multisets $\nu$  of subsets of $[m-1]$.
\end{theorem}  
For instance, for $m=2$, and with $\nu_0$ for $\nu_\emptyset$ and $\nu_1$ for $\nu_{\set{1}}$, we   get
\[
\frac{L_{2,n}}{n!}
=
(-1)^n \sum_{\nu_0 + \nu_1=n} (-1)^{\nu_1} \binom{n}{\nu_0, \nu_1} (\nu_1)^{\nu_0} (\nu_1-1)^{\nu_1}.
\]
Writing $k$ for $\nu_1$ and $n-k$ for $\nu_0$, we get:
\[
\frac{L_{2,n}}{n!}
=
(-1)^n \sum_{k=0}^n (-1)^k \binom{n}{k} k^{n-k} (k-1)^k.
\]
As pointed out in \cite{Doyle}, this is a formula for the  derangement numbers due to Ryser \cite{Ryser}.

For $m=3$, we get
  \[
  \frac{L_{3,n}}{n!}= \sum_{\nu}
  (-1)^{\nu_1+\nu_2} \binom{n}{\nu_0, \nu_1, \nu_2, \nu_{12}}
  f_0^{\nu_0}
  f_1^{\nu_1}
  f_2^{\nu_2}
  f_{12}^{\nu_{12}},   
  \]
      where the sum is over all tuples of nonnegative integers
    $
    \nu=(\nu_0, \nu_1, \nu_2, \nu_{12})
    $
    with sum $n$, and
    \begin{align*}
    f_0 &= (\nu_1+\nu_{12})(\nu_2+\nu_{12})-\nu_{12},\\
   f_1 &= (\nu_1-1+\nu_{12})(\nu_2+\nu_{12})-\nu_{12}, \\
   f_2 &= (\nu_1+\nu_{12})(\nu_2-1+\nu_{12})-\nu_{12},\\
  f_{12}&= (\nu_1+\nu_{12}-1)(\nu_2+\nu_{12}-1)-(\nu_{12}-1).
    \end{align*}

\subsection{Proof of Doyle's First Formula with multisymmetric functions}
    
    We start as in Gessel's Formula : the arrays fulfilling column conditions are those corresponding to sequences of summands of $e([m])$. Those fulfilling  the new row conditions (``every row contains \emph{at least} one occurrence of each symbol'')  are those whose product is of the form 
  $\multi{x}_1^{\vec{u}_1} \multi{x}_2^{\vec{u}_2} \cdots \multi{x}_r^{\vec{u}_r}$
    where the $\vec{u}_i$ are vectors whose components are all non-zero. After Lemma \ref{extraction}, we can conclude that
    \[
    L_{m,n} = \scalar{e([m])^n}{ \NZn}
    ,
    \quad \text{ where }
    \NZ  = \sum_{\alpha} \iverson{S(\alpha) = [m]}  \cdot  h_{\alpha}
    .
  \]
  Following Doyle, we use now inclusion-exclusion:
  \[
  \forall \alpha\in \NN^m,
  \qquad
  \iverson{S(\alpha) = [m]} 
   = \sum_{B \subset [m]} (-1)^{m-\#B}
   \iverson{S(\alpha) \subset B}.
  \]
  From this, we deduce
  \[
  \NZ  = \sum_{B \subset [m]} (-1)^{m-\#B}\;\sigma_B
  , \quad \text{ where }
  \sigma_B = \sum_{\alpha}\iverson{S(\alpha) \subset B} \cdot h_{\alpha}
 . 
\]
Now we have:
\begin{align*}
L_{m,n}
&=
\scalar{e([m])^n}{\NZn}
              \\
&=
\scalar{e([m])^n}{
  \left(    \sum_{B \subset [m]} (-1)^{m-\#B} \sigma_B    \right)^n
}\\
&=
\scalar{  e([m])^n}{
  (-1)^{mn} \sum_{\nu}
 \binom{n}{\nu} 
   \prod_{B \subset [m]} \left((-1)^{\#B} \sigma_B\right)^{\nu_B}
}\\
  &=(-1)^{mn}
  \cdot
  \sum_{\nu}
  (-1)^{\|\nu\|} \binom{n}{\nu} 
\scalar{  e([m])^n}{
   \prod_{B \subset [m]}  \sigma_B^{\nu_B}
}
.
\end{align*}
We will now ``move each $\sigma_B$ to the other side of the scalar product'' by introducing the adjoint $\sigma_B^\perp$ of the operator of multiplication by $\sigma_B$.  Let us first elucidate the effect this adjoint.

We start with the effect of a more general operator. Remember that in the theory of symmetric functions, the adjoint $\sigma_z^\perp$ of the series of the complete functions $\sigma_z=\sum_{n} h_n z^n$ sends each power sum $p_k$ to $p_k+z^k$. This generalizes  straightforwardly for multisymmetric functions.
\begin{proposition} \label{sigma t}
  Let
  \[
  \sigma(\multi{t})=\sigma(t_1, t_2, \ldots, t_m) = \sum_{\alpha} h_{\alpha} \multi{t}^\alpha.
  \]
  Then its adjoint operator $\sigma(\multi{t})^\perp$ of the multiplication by $\sigma(\multi{t})$ is the morphism of algebra that sends every power sum $p_{\alpha}$ to $p_{\alpha}+ \multi{t}^{\alpha}$.
\end{proposition}  
\begin{proof}
  From the formulas for the scalar products of products of power sum, one can check that
  \[
  \left(\frac{p_\alpha}{Z_\alpha}\right)^\perp = \frac{\partial \phantom{p_{\alpha}}}{\partial {p_{\alpha}}}.
  \]
  Since
  \[
  \sigma(\multi{t}) =
  \exp\left( \sum_{\alpha \neq \multi{0}} \frac{p_\alpha}{Z_\alpha} \multi{t}^\alpha\right),
  \]
  we have
  \[
  \sigma(\multi{t})^\perp =
  \exp\left( \sum_{\alpha \neq \multi{0}}  \multi{t}^\alpha \frac{\partial \phantom{p_{\alpha}}}{\partial {p_{\alpha}}} \right).
  \]
  One can check that this sends  any product of power sums $\prod_{\alpha\neq \multi{0}} p_{\alpha}^{\ell_\alpha}$ to 
  $\prod_{\alpha\neq \multi{0}} (p_\alpha+\multi{t}^\alpha)^{\ell_\alpha}$.
  \end{proof}

\begin{corollary}\label{sigma B}
  Let $B\subset [m]$.
  The adjoint $\sigma_B^\perp$ of the multiplication by $\sigma_B$ is the morphism of algebras such that,  
for any power sum $p_{\alpha}$,
\[
\sigma_B^{\perp}(p_{\alpha})
=
\begin{cases}
  p_{\alpha}+1 & \text{ if } S(\alpha) \subset B,\\
  p_{\alpha}   & \text{ otherwise.}
\end{cases}  
\]
\end{corollary}
\begin{proof}
  This follows from Proposition \ref{sigma t}, after observing that  $\sigma_B$ is obtained from $\sigma(\multi{t})$ by setting $t_i=[i \in B]$.
\end{proof}

In particular, if we apply $\sigma_B^\perp$ to $A=p_{\chi(A)}$, then we get
\[
\sigma_B^{\perp}(A) = \begin{cases}
  A+1 & \text{ if } A \subset B,\\
  B   & \text{ otherwise.}
\end{cases}  
\]

\begin{corollary}\label{sigma B power}
  Let $\nu$ be a multiset of  subsets of $[m]$.
 
  \[
  \forall f \in \FA,
  \qquad
  \scalar{f}{\prod_{B \subset [m]}  \sigma_B^{\nu_B}
}  =
  f_{| A \leftarrow \nu^+_A}.
  \]
\end{corollary}
\begin{proof}
  Taking adjoints, we see that
  \[
  \scalar{f}{\prod_{B \subset [m]}  \sigma_B^{\nu_B}
}  =
  \scalar{\prod_{B \subset [m]}  \left(\sigma_B^\perp\right)^{\nu_B} \left(f\right)}
         {1}
         .
  \]
The effect of each $\sigma_B^\perp$ is to replace, in its argument, each set $A$ contained in $B$ with $A+1$.
Thus  $\left(\sigma_B^\perp\right)^{\nu_B}$ replaces each $A\subset B$ with $A+\nu_B$, and $\prod_{B \subset [m]}  \left(\sigma_B^\perp\right)^{\nu_B}$ replaces each $A$ with
\[
A + \sum_{B: \atop A\subset B} \nu_B.
\]
which is $A+\nu^+_A$.
Last, the scalar product with $1$ kills all variables. The overall effect is to replace each 
$A$ with $\nu^+_A$.
\end{proof}

This concludes the proof of Doyle's First Formula.

\subsection{Proof of Doyle's Second Formula with multisymmetric functions}

In the calculation that follows, it will be convenient to use some new notations. For $\beta=(\beta_1, \ldots, \beta_{m-1})$, we denote with $\beta,k$ the concatenation $(\beta_1, \ldots, \beta_{m-1}, k)$.  We will also use $E$ for $e([m])$.

Doyle proved the formula by starting considering only the normalised Latin rectangles (those whose first line is $1,2,\ldots,n$). This is hardly compatible with our methods since it breaks symmetries. We will start slightly differently.

We replace the row conditions for Latin rectangles with:
\begin{itemize}
\item in the last row, every symbol appears \emph{exactly} one.
\item in every other row, every symbol appears \emph{at most} once.
\end{itemize}

We deduce that $L_{m,n}$ is the sum of the coefficients, in the expansion of $E^n$ (remember that, in this section,  $E$ stands for $e([m])$),  of the monomials in the $x_{i,j}$ ($i \le n$) such that each $x_{i,j}$ appears at least once if $j<m$, and exactly once if $j=m$.
Applying Lemma \ref{extraction}, we get
\[
L_{m,n} = \scalar{E^n}{K^n}
  ,
  \quad \text{ where }
  K = \sum_{\beta}   \iverson{S(\beta)=[m-1]} \cdot h_{\beta,1}
  .
\]
where, here and in what follows, the  indices $\beta$ run in $\NN^{m-1 }$.

Applying once again inclusion-exclusion, we have:
\[
K
=
\sum_{B \subset [m-1]} (-1)^{m-1-\#B}
\;  K_B
, \quad \text{ where }
K_B =  \sum_\beta \iverson{S(\beta)\subset B} \cdot h_{\beta,1}
.
\]
Expanding $K^n$ in $L_{m,n}=\scalar{E^n}{K^n}$, we get
\begin{align*}
  L_{m,n}
  &= 
\scalar{E^n}{
\sum_\nu \binom{n}{\nu} \prod_{B \subset [m-1]} \left((-1)^{m-1-\#B} K_B\right)^{\nu_B} 
}
\\
&=
(-1)^{mn-n} \sum_\nu  (-1)^{\|\nu \|} \binom{n}{\nu}
 \scalar{E^n}{
\prod_{B \subset [m-1]} K_B^{\nu_B} 
}
\end{align*}
where the sum is over all $n$-multisets $\nu$ of subsets of $[m-1]$.   

We will now determine the operators $K_B^\perp$. To this aim, let us introduce the following 1-parameter deformation:
 \[
\KK_B(z) =   \sum_{\beta, k}   \iverson{S(\beta)\subset B} \cdot h_{\beta,k} z^k 
\]
so that $K_B = \KK_B'(0)$.

\begin{lemma}\label{decomposition}
  We have
  \[
  K_B = \sigma_B \cdot \delta_B , \qquad  \text{ where }\quad
  \delta_B = \sum_{\beta} \iverson{S(\beta)\subset B} \frac{p_{\beta, 1}}{Z_{\beta,1}}.
  \]
\end{lemma}

\begin{proof}
  We observe that $\KK_B(z)$ is obtained from $\sigma(\multi{t})$ by replacing $t_i$ with $z$ if $i=m$ , with $1$  if $i\in B$ and with $0$ otherwise.
  By applying this specialization to 
  \[
  \sigma(\multi{t}) =\exp\left(\sum_{\alpha\neq \multi{0}} \frac{p_{\alpha}}{Z_\alpha} \multi{t}^\alpha\right),
  \]
we get
\[
\KK_B(z)=
\exp\left(\sum_{(\beta, k) \neq (\multi{0}, 0)} \iverson{S(\beta)\subset B} \frac{p_{\beta,k}}{Z_{\beta, k}} z^k\right)
\]
that we can now differentiate with respect to $z$ and evaluate at $z=0$, this gives the lemma. 
\end{proof}
We have now
\begin{align*}
  L_{m,n}
  &=
  (-1)^{mn-n} \sum_\nu  (-1)^{\|\nu \|} \binom{n}{\nu}
\scalar{E^n} {
  \prod_{B \subset [m-1]}( \sigma_B)^{\nu_B}  \prod_{B\subset [m-1]} (\delta_B)^{\nu_B}
}
\\
&=
 (-1)^{mn-n} \sum_\nu  (-1)^{\|\nu \|} \binom{n}{\nu}
\scalar{\left(\prod_{B\subset [m-1]} (\delta_B)^{\nu_B}\right)^\perp \left(E^n\right)} {
  \prod_{B \subset [m-1]}( \sigma_B)^{\nu_B}  
}
\end{align*}

Note that $\delta_B^\perp$ acts on $\FA$ as the first-order differential operator $\sum_{A} [A \subset B] \frac{\partial \phantom{\partial A\cup \set{m}}}{\partial A\cup \set{m}}$, and the formula requires now the evaluation of a composition of  $n$ such differential operators on $E^n$. But, since each term in $E$ depends on only one $A$ containing $m$, in degree $1$, the polynomial  $\delta_B^\perp (E)$ is independent from all variables $A$ with $m\in A$, and thus, $\delta_C^\perp \circ \delta_B^\perp(E)$ is always $0$.  This allows to establish, by induction, that for all $k\le n$, and any
sequence of $n$ operators $\delta_{B_1}^\perp$, $\delta_{B_2}^\perp$, \ldots ,  $\delta_{B_n}^\perp$,
\[
\delta_{B_k}^\perp \circ \delta_{B_{k-1}}^\perp \circ \cdots \circ \delta_{B_1}^\perp (E^n)
=
(n)_k\cdot E^{n-k} \cdot \delta_{B_k}^\perp (E) \cdot \delta_{B_{k-1}}^\perp(E) \cdot \cdots \cdot \delta_{B_1}^\perp (E). 
\]
In particular, for $k=n$,  
\[
\left(\prod_{B\subset [m-1]} (\delta_B)^{\nu_B}\right)^\perp \left(E^n\right)
=
n! \prod_{B\subset [m-1]} (\delta_B^\perp(E))^{\nu_B} .
\]
Thus,
\[
\frac{L_{m,n}}{n!}
= (-1)^{mn-n} \sum_\nu  (-1)^{\|\nu \|} \binom{n}{\nu}
\scalar{  \prod_{B\subset [m-1]} \left(\delta_B^\perp(E) \right)^{\nu_B} }
       {
         \prod_{B\subset [m-1]} (\sigma_B)^{\nu_B}.
         }
\]
It remains to determine $\delta_B^\perp(E)$.

\begin{lemma}\label{killer_B perp}
  \[
  \delta_B^\perp(e([m])) = (\sigma_B^\perp)^{-1}(e([m-1]))
  \]
\end{lemma}

\begin{proof}
  Since $\sigma_B^\perp \circ \delta_B^\perp=K_B^\perp$, this amounts to show that $K_B^\perp (e([m])) = e([m-1])$.

  Given two morphisms of algebras $\XX$ and $\YY$ from the algebra of multisymmetric functions to another algebra, define their sum $\XX +_p \YY$ as the unique morphism of algebras fulfilling, $(\XX +_p \YY)(p_\alpha) = \XX(p_\alpha) + \YY(p_\alpha)$  for any power sum $p_\alpha$ (we are generalizing here to multisymmetric functions the classical \emph{sum of alphabets} of symmetric functions).
  Then, it is easy to see that for any elementary function $e_\alpha$,
  \[
  (\XX +_p \YY)(e_\alpha) = \sum_{\beta, \gamma\atop \beta+\gamma=\alpha} \XX(e_\beta) \YY(e_\gamma).
  \]
  Apply this for $\XX$ equal to the identity, and $\YY$ the morphism obtained by replacing $x_{1,m}$ with $z$,  $x_{1,j}$ with $1$ if $j\in B$ and all other $x_{i,j}$ with $0$. Then,  $\XX+_p\YY=\KK_B(z)^\perp$. Therefore,
  \[
  \KK_B(z)^\perp(e_\alpha)
  =
  \sum_{\beta, \gamma\atop \beta+\gamma=\alpha} e_\beta \YY(e_\gamma).
  \]
  It is easy to see (from instance, applying $\YY$ to the generating series of the $e_\alpha$) that $\YY(e_\beta)$ is $z$ if $\beta=\chi_m$, $1$ if $\beta$ is the zero vector or $\beta$ is $\chi_j$ for some $j\in B$, and $0$ otherwise. Therefore,
  \[
  \KK_B(z)^\perp(e_\alpha) = e_\alpha + \sum_{j\in B} e_{\alpha-\chi_j}  +  e_{\alpha-\chi_m} \cdot z.
  \]
  where $e_\alpha$ stands for $0$ when $\alpha$ has some negative component.
  Using $K_B=\KK_B'(0)$, we get $K_B^\perp(e_\alpha) = e_{\alpha-\chi_m}$. In particular, $K_B^\perp(e([m]))=e([m-1])$.
\end{proof}  
We can now finish the proof of Doyle's  Second Formula. For each $n$-multiset $\nu$ of subsets of $[m]$, we have
\begin{align*}
& \scalar{  \prod_{B\subset [m-1]} \left(\delta_B^\perp(E) \right)^{\nu_B} }
       {
         \prod_{B\subset [m-1]} (\sigma_B)^{\nu_B}
       }
       \\
       =&
\scalar{  \prod_{B\subset [m-1]} \left((\sigma_B^{-1})^{\perp}(e([m-1])) \right)^{\nu_B} }
       {
         \prod_{B\subset [m-1]} (\sigma_B)^{\nu_B}
       }\\
       =&
\scalar{  \prod_{B\subset [m-1]} \left(e([m-1])_{|A \leftarrow A - [A \subset B]} \right)^{\nu_B} }
       {
         \prod_{B\subset [m-1]} (\sigma_B)^{\nu_B}
       }\\
       =&
  \prod_{B\subset [m-1]} \left(e([m-1])_{|A \leftarrow \nu^+_A - [A \subset B]} \right)^{\nu_B} .
\end{align*}


\end{document}